\documentclass[letterpaper, 10pt, conference]{ieeeconf}
\IEEEoverridecommandlockouts 
\usepackage{amsmath,amssymb,url}
\usepackage{graphicx,subfigure,hyperref}
\usepackage{color}

\newcommand{\bracket}[1]{\ensuremath{\left[ #1 \right]}}
\newcommand{\braces}[1]{\ensuremath{\left\{ #1 \right\}}}

\newcommand{\refeqn}[1]{(\ref{eqn:#1})}
\newcommand{\reffig}[1]{Fig. \ref{fig:#1}}
\newcommand{\tr}[1]{\mbox{tr}\ensuremath{\negthickspace\bracket{#1}}}
\newcommand{\trs}[1]{\mbox{tr}\ensuremath{\!\bracket{#1}}}

\newcommand{\SO}{\ensuremath{\mathsf{SO(3)}}}

\newcommand{\so}{\ensuremath{\mathfrak{so}(3)}}

\renewcommand{\Re}{\ensuremath{\mathbb{R}}}
\newcommand{\Sph}{\ensuremath{\mathsf{S}}}

\newcommand{\D}{\ensuremath{\mathbf{D}}}
\renewcommand{\d}{\ensuremath{\mathbf{d}}}

\title{\LARGE \bf
Geometric Control of Multiple Quadrotor UAVs\\ Transporting a Cable-Suspended Rigid Body}

\author{Taeyoung Lee
\thanks{Taeyoung Lee, Mechanical and Aerospace Engineering, George Washington University, Washington, DC 20052 {\tt tylee@gwu.edu}}%
\thanks{This research has been supported in part by NSF under the grant CMMI-1243000 (transferred from 1029551), CMMI-1335008, and CNS-1337722.}
}

\newtheorem{prop}{Proposition}
\newtheorem{remark}{Remark}

\begin{document}
\allowdisplaybreaks

\maketitle \thispagestyle{empty} \pagestyle{empty}

\begin{abstract}
This paper is focused on tracking control for a rigid body payload, that is connected to an arbitrary number of quadrotor unmanned aerial vehicles via rigid links. An intrinsic form of the equations of motion is derived on the nonlinear configuration manifold, and a geometric controller is constructed such that the payload asymptotically follows a given desired trajectory for its position and attitude. The unique feature is that the coupled dynamics between the rigid body payload, links, and quadrotors are explicitly incorporated into control system design and stability analysis. These are developed in a coordinate-free fashion to avoid singularities and complexities that are associated with local parameterizations. The desirable features of the proposed control system are illustrated by a numerical example.
\end{abstract}

\section{Introduction}


Aerial transport of payloads by towed cables is common in various situations, such as emergency response, industrial, and military applications. Examples of aerial towing range from emergency rescue missions where individuals are lifted from dangerous situations to the delivery of heavy equipment to the top of a tall building.

Transportation of a cable-suspended load has been studied traditionally for helicopters~\cite{CicKanJAHS95,BerPICRA09}. Small unmanned aerial vehicles or quadrotors are also considered for load transportation and deployments~\cite{PalCruIRAM12,MicFinAR11,MazKonJIRS10}. However, these are based on simplified dynamics models. For example, the effects of the payload are considered as additional force and torque exerted to quadrotors, instead of considering the dynamic coupling between the payload and the quadrotor, and a pre-computed trajectory that minimizes swing motion of the payload is followed, instead of actively controlling the motion of payload and cable~\cite{MicFinAR11}. As such, these may not be suitable for agile load transportation where the motion of cable and payload should be actively suppressed online. 

Recently, geometric nonlinear control systems are developed for the complete dynamic model of a single quadrotor transporting a cable-suspended load~\cite{SreLeePICDC13}, and for multiple quadrotors transporting a common payload cooperatively~\cite{LeeSrePICDC13}. It is also generalized for a quadrotor with a payload connected by flexible cable that is modeled as a serially-connected links, to incorporate the deformation of cable~\cite{GooLeePACC14}. However, in these results, it is assumed that the payload is modeled by a point mass. Such assumption is quite restrictive for practical cases where the size of the payload is comparable to the quadrotors and the length of cables.

The objective of this paper is to construct a control system for an arbitrary number of quadrotors connected to a rigid body payload via rigid links. This is challenging in the sense that \textit{dynamically coupled} quadrotors should cooperate safely to transport a rigid body. This is in contrast to the existing results on formation control of decoupled multi-agent systems.

\begin{figure}
\centerline{\setlength{\unitlength}{0.1\columnwidth}\footnotesize
\begin{picture}(10,6)(0.0,0)
\put(1.25,0){\includegraphics[width=0.75\columnwidth]{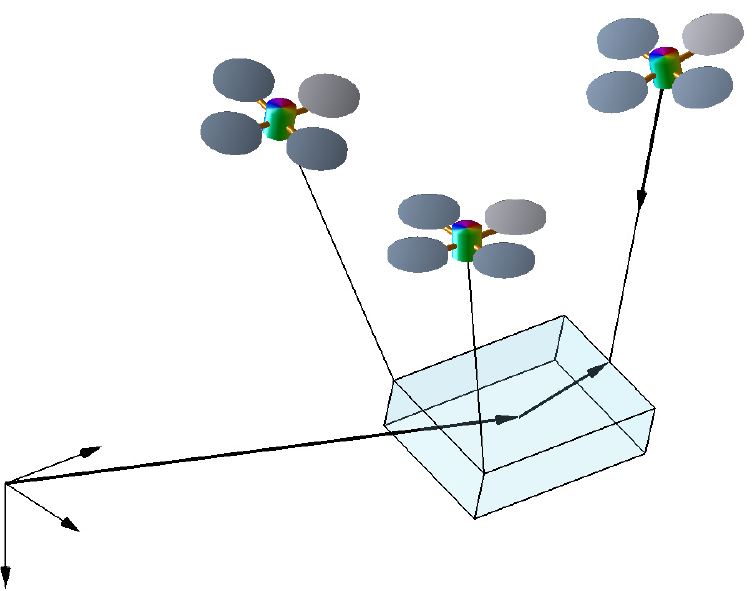}}
\put(2.35,1.45){$\vec e_1$}
\put(1.9,0.3){$\vec e_2$}
\put(1.4,-0.1){$\vec e_3$}
\put(3.1,1.0){$x_0\in\Re^3$}
\put(6.7,0.5){\shortstack[c]{$m_0,J_0$\\$R_0\in\SO$}}
\put(5.3,4.9){\shortstack[c]{$m_i,J_i$\\$R_i\in\SO$}}
\put(7.8,4.2){$q_i\in\Sph^2$}
\put(6.8,1.8){$\rho_i$}
\put(7.65,3.3){$l_i$}
\end{picture}}
\caption{Dynamics model: $n$ quadrotors are connect to a rigid body $m_0$ via massless links $l_i$. The configuration manifold is $\Re^3\times\SO\times(\Sph^2\times\SO)^n$.}\label{fig:DM}
\vspace*{-0.1cm}
\end{figure}
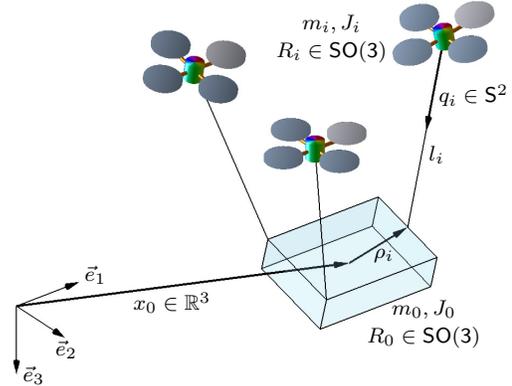

In this paper, a coordinate-free form of the equations of motion is derived according to Lagrange mechanics on a nonlinear manifold, and a geometric control system is designed such that the rigid body payload exponentially follows a given desired trajectory of both the payload position and attitude. The unique property of the proposed control system is that the nontrivial coupling effects between the dynamics of payload, cable, and multiple quadrotors are explicitly incorporated into control system design, without any simplifying assumption. Another distinct feature is that the equations of motion and the control systems are developed directly on the nonlinear configuration manifold intrinsically. Therefore, singularities of local parameterization are completely avoided to generate agile maneuvers of the payload in a uniform way. In short, the proposed control system is particularly useful for rapid and safe payload transportation in complex terrain, where the position and attitude of the payload should be controlled concurrently in a fast manner to avoid collision with obstacles. 

This paper is organized as follows. A dynamic model is presented and the problem is formulated at Section \ref{sec:DM}. Control systems are constructed at Sections \ref{sec:SDM} and \ref{sec:FDM}, which are followed by a numerical example. Due to the page limit, parts of proofs are relegated to~\cite{LeeArXiv14a}.

\section{Problem Formulation}\label{sec:DM}

Consider $n$ quadrotor UAVs that are connected to a payload, that is modeled as a rigid body, via massless links (see Figure \ref{fig:DM}). Throughout this paper, the variables related to the payload is denoted by the subscript $0$, and the variables for the $i$-th quadrotor are denoted by the subscript $i$, which is assumed to be an element of $\mathcal{I}=\{1,\cdots\, n\}$ if not specified. We choose an inertial reference frame $\{\vec e_1,\vec e_2,\vec e_3\}$ and body-fixed frames $\{\vec b_{j_1},\vec b_{j_2},\vec b_{j_3}\}$ for $0\leq j\leq n$ as follows. For the inertial frame, the third axis $\vec e_3$ points downward along the gravity and the other axes are chosen to form an orthonormal frame. The origin of the $j$-th body-fixed frame is located at the center of mass of the payload for $j=0$ and at the mass center the quadrotor for $1\leq j\leq n$. The third body-fixed axis $\vec b_{i_3}$ is normal to the plane defined by the centers of rotors, and it points downward.

The location of the mass center of the payload is denoted by $x_0\in\Re^3$, and its attitude is given by $R_0\in\SO$, where the special orthogonal group is defined by $\SO=\{R\in\Re^{3\times 3}\,|\, R^TR=I,\,\mathrm{det}[R]=1\}$. Let $\rho_i\in\Re^3$ be the point on the payload where the $i$-th link is attached, and it is represented with respect to the zeroth body-fixed frame. The other end of the link is attached to the mass center of the $i$-th quadrotor. The direction of the link from the mass center of the $i$-th quadrotor toward the payload is defined by the unit-vector $q_i\in\Sph^2$, where $\Sph^2=\{q\in\Re^3\,|\,\|q\|=1\}$, and the length of the $i$-th link is denoted by $l_i\in\Re$.

Let $x_i\in\Re^3$ be the location of the mass center of the $i$-th quadrotor with respect to the inertial frame. As the link is assumed to be rigid, we have $x_i=x_0+R_0\rho_i-l_i q_i$. The attitude of the $i$-th quadrotor is defined by $R_i\in\SO$, which represents the linear transformation of the representation of a vector from the $i$-th body-fixed frame to the inertial frame. The corresponding configuration manifold of this system is $\Re^3\times\SO\times(\Sph^2\times \SO)^n$.

The mass and the inertia matrix of the payload are denoted by $m_0\in\Re$ and $J_0\in\Re^{3\times 3}$, respectively. The dynamic model of each quadrotor is identical to~\cite{LeeLeoPICDC10}. The mass and the inertia matrix of the $i$-th quadrotor are denoted by $m_i\in\Re$ and $J_i\in\Re^{3\times 3}$, respectively. The $i$-th quadrotor can generates a thrust $-f_iR_ie_3\in\Re^3$ with respect to the inertial frame, where $f_i\in\Re$ is the total thrust magnitude and $e_3=[0,0,1]^T\in\Re^3$. It also generates a moment $M_i\in\Re^3$ with respect to its body-fixed frame. The control input of this system corresponds to $\{f_i,M_i\}_{1\leq i\leq n}$.

Throughout this paper, the 2-norm of a matrix $A$ is denoted by $\|A\|$, and its maximum eigenvalue and minimum eigenvalues are denoted by $\lambda_{M}[A]$ and $\lambda_{m}[A]$, respectively. The standard dot product is denoted by $x \cdot y = x^Ty$ for any $x,y\in\Re^3$.

\subsection{Equations of Motion}

The kinematic equations for the payload, quadrotors, and links are given by
\begin{gather}
\dot q_{i} = \omega_i\times q_i=\hat\omega_i q_i,\label{eqn:dotqi}\\
\dot R_0  = R_0\hat\Omega_0,\quad \dot R_i  = R_i\hat\Omega_i,\label{eqn:dotRi}
\end{gather}
where $\omega_i\in\Re^3$ is the angular velocity of the $i$-th link, satisfying $q_i\cdot\omega_i=0$, and $\Omega_0$ and $\Omega_i\in\Re^3$ are the angular velocities of the payload and the $i$-th quadrotor expressed with respect to its body-fixed frame, respectively. The \textit{hat map} $\hat\cdot:\Re^3\rightarrow\so$ is defined by the condition that $\hat x y=x\times y$ for all $x,y\in\Re^3$, and the inverse of the hat map is denoted by the \textit{vee map} $\vee:\so\rightarrow\Re^3$. 

We derive equations of motion according to Lagrangian mechanics. The velocity of the $i$-th quadrotor is given by $\dot x_i = \dot x_0+\dot R_0\rho_i - l_i\dot q_i$. The kinetic energy of the system is composed of the translational kinetic energy and the rotational kinetic energy of the payload and quadrotors:
\begin{align}
\mathcal{T} & = \frac{1}{2}m_0 \|\dot x_0\|^2 + \frac{1}{2}\Omega_0\cdot J_0\Omega_0\nonumber\\
&\quad +\sum_{i=1}^n \frac{1}{2} m_i \|\dot x_0+\dot R_0\rho_i - l_i\dot q_i\|^2+ \frac{1}{2}\Omega_i\cdot J_i\Omega_i.\label{eqn:TT}
\end{align}
The gravitational potential energy is given by
\begin{align}
\mathcal{U} = - m_0g e_3 \cdot x_0 - \sum_{i=1}^n m_ige_3\cdot (x_0+R_0\rho_i-l_iq_i),\label{eqn:UU}
\end{align}
where it is assumed that the unit-vector $e_3$ points downward along the gravitational acceleration as shown at \reffig{DM}. The corresponding Lagrangian of the system is $\mathcal{L}=\mathcal{T}-\mathcal{U}$. 

Coordinate-free form of Lagrangian mechanics on the two-sphere $\Sph^2$ and the special orthogonal group $\SO$ for various multibody systems has been studied in~\cite{Lee08,LeeLeoIJNME08}. The key idea is representing the infinitesimal variation of $q_i\in\Sph^2$ in terms of the exponential map:
\begin{align}
\delta q_i = \frac{d}{d\epsilon}\bigg|_{\epsilon = 0} \exp (\epsilon \hat\xi_i) q_i = \xi_i\times q_i,\label{eqn:delqi}
\end{align}
for a vector $\xi_i\in\Re^3$ with $\xi_i\cdot q_i=0$. Similarly, the variation of $R_i$ is given by $\delta R_i  = R_i\hat\eta_i$ for $\eta_i\in\Re^3$. 

By using these expressions, the equations of motion can be obtained from Hamilton's principle as follows (see Appendix \ref{sec:EOM} for more detailed derivations).
\begin{gather}
M_q(\ddot x_0-ge_3) - \sum_{i=1}^n m_iq_iq_i^T R_0\hat\rho_i\dot\Omega_0\nonumber\\
  = \sum_{i=1}^n u_i^\parallel-m_il_i\|\omega_i\|^2q_i- m_iq_iq_i^T R_0\hat\Omega_0^2\rho_i,\label{eqn:ddotx0}\\
  (J_0-\sum_{i=1}^n m_i\hat\rho_i R_0^T q_iq_i^T R_0\hat\rho_i)\dot\Omega_0 
+ \sum_{i=1}^n m_i\hat\rho_i R_0^Tq_iq_i^T(\ddot x_0-ge_3)\nonumber \\ + \hat\Omega_0   J_0\Omega_0
=\sum_{i=1}^n \hat\rho_i R_0^T (u_i^\parallel-m_il_i\|\omega_i\|^2 q_i-m_iq_iq_i^TR_0\hat\Omega_0^2\rho_i),\label{eqn:dotW0}\\
\dot\omega_i =\frac{1}{l_i}\hat q_i(\ddot x_0-ge_3 - R_0\hat\rho_i\dot\Omega_0+
R_0\hat\Omega_0^2\rho_i)
 -\frac{1}{m_il_i}\hat q_iu_i^\perp,\label{eqn:dotwi}\\
J_i\dot\Omega_i + \Omega_i\times J_i\Omega_i = M_i,\label{eqn:dotWi}
\end{gather}
where $M_q=m_y I + \sum_{i=1}^n m_i q_i q_i^T\in\Re^{3\times 3}$, which is symmetric, positive-definite for any $q_i$.

The vector $u_i\in\Re^3$ represents the control force at the $i$-th quadrotor, i.e., $u_i=-f_i R_i e_3$.  The vectors $u_i^\parallel$ and $u_i^\perp\in\Re^3$ denote the orthogonal projection of $u_i$ along $q_i$, and the orthogonal projection of $u_i$ to the plane  normal to $q_i$, respectively, i.e.,
\begin{align}
u_i^\parallel &= (I+\hat q_i^2) u_i = (q_i\cdot u_i) q_i = q_iq_i^T u_i,\\
u_i^\perp &= -\hat q_i^2 u_i = -q_i \times (q_i \times u_i)=(I-q_iq_i^T) u_i.
\end{align}
Therefore, $u_i = u_i^\parallel + u_i^\perp$.

\subsection{Tracking Problem}

Define a matrix $\mathcal{P} \in \Re^{6\times 3n}$ as 
\begin{align}
\mathcal{P} = \begin{bmatrix} I_{3\times 3} & \cdots & I_{3\times 3} \\ \hat\rho_1 & \cdots & \hat\rho_n \end{bmatrix}. \label{eqn:P}
\end{align}
Assume the links are attached to the payload such that
\begin{align}
\mathrm{rank}[\mathcal{P}] \geq 6.\label{eqn:A1}
\end{align}
This is to guarantee that there exist enough degrees of freedom in control inputs for both the translational motion and the rotational maneuver of the payload. The assumption \refeqn{A1} requires that the number of quadrotor is at least three, i.e., $n\geq 3$, since when $n=2$ the above matrix $\mathcal{P}$ has a non-empty null space spanned by $[(\rho_1-\rho_2)^T, (\rho_2-\rho_1)^T]^T$. This follows from the fact that it is impossible to generate any moment to the payload along the direction of $\rho_1-\rho_2$ when $n=2$. 

Suppose that the desired trajectories for the position and the attitude of the payload are given as smooth curves, namely $x_{0_d}(t)\in\Re^3$ and $R_{0_d}(t)\in\SO$ during a time period. From the attitude kinematics equation, we have
\begin{align}
\dot R_{0_d} = R_{0_d} \hat\Omega_{0_d},
\end{align}
where $\Omega_{0_d}\in\Re^3$ corresponds to the desired angular velocity. It is assumed that the velocity and the acceleration of the desired trajectories are bounded by known constants.

We wish to design a control input of each quadrotor $\{f_i,M_i\}_{1\leq i\leq n}$ such that the state of zero tracking errors becomes an asymptotically stable equilibrium of the controlled system. 

\section{Control System Design For Simplified Dynamic Model}\label{sec:SDM}

In this section, we consider a simplified dynamic model where the attitude dynamics of each quadrotor is ignored, and we design a control input by assuming that the thrust at each quadrotor, namely $u_i$ can be arbitrarily chosen. It corresponds to the case where each quadrotor is replaced by a fully actuated vehicle that can generates a thrust arbitrarily. The effects of the attitude dynamics of quadrotors will be incorporated in the next section. 

In the simplified dynamic model given by \refeqn{ddotx0}-\refeqn{dotwi}, the dynamics of the payload are affected by the parallel components $u_i^\parallel$ of the control inputs, and the dynamics of the links are directly affected by the normal components $u_i^\perp$ of the control inputs. This motivates the following control system design procedure: first, the parallel components $u_i^\parallel$ are chosen such that the payload follows the desired position and attitude trajectory while yielding the desired direction of each link, namely $q_{i_d}$; next, the normal components $u_i^\perp$ are designed such that the actual direction of the links $q_i$ follows $q_{i_d}$.

\subsection{Design of Parallel Components}

Let $a_i\in\Re^3$ be the acceleration of the point on the payload where the $i$-th link is attached, relative to the gravitational acceleration:
\begin{align}
a_i = \ddot x_0 - ge_3 + R_0\hat\Omega_0^2\rho_i -R_0\hat\rho_i\dot\Omega_0.\label{eqn:ai0}
\end{align}
The parallel component of the control input is chosen as
\begin{align}
u_i^\parallel & = \mu_i + m_il_i\|\omega_i\|^2q_i + m_i q_iq_i^T a_i,\label{eqn:uip}
\end{align}
where $\mu_i\in\Re^3$ is a virtual control input that is designed later, with a constraint that $\mu_i$ is parallel to $q_i$. Note that the expression of $u_i^\parallel$ is guaranteed to be parallel to $q_i$ due to the projection operator $q_iq_i^T$ at the last term of the right-hand side of the above expression.

The motivation for the proposed parallel components becomes clear if \refeqn{uip} is substituted into \refeqn{ddotx0}-\refeqn{dotW0} and rearranged to obtain
\begin{gather}
m_0 (\ddot x_0 -g e_3) = \sum_{i=1}^n \mu_i,\label{eqn:ddotx0s}\\
J_0\dot\Omega_0 +\hat\Omega_0 J_0\Omega_0 = \sum_{i=1}^n \hat\rho_i R_0^T \mu_i.\label{eqn:dotW0s}
\end{gather}
Therefore, considering a free-body diagram of the payload, it is clear that the virtual control input $\mu_i$ corresponds to the force exerted to the payload by the $i$-link, namely the tension of the $i$-th link. When there is no control force from each quadrotor, i.e., $u_i^\parallel=0$, the tension of the $i$-th link is composed of the projected relative inertial force at the point where the $i$-th link is attached to the payload and the centrifugal force due to the rotation of the link. Substituting \refeqn{ddotx0s} and \refeqn{dotW0s} back into \refeqn{ai0}, we obtain
\begin{align}
a_i & = \frac{1}{m_0}\sum_{j=1}^n\mu_j+R_0\hat\Omega_0^2\rho_i\nonumber\\
&\quad + R_0\hat\rho_iJ_0^{-1}(\hat\Omega_0 J_0\Omega_0 - \sum_{j=1}^n \hat\rho_j R_0^T\mu_j).\label{eqn:ai}
\end{align}

Next, we determine the virtual control input $\mu_i$. Any control scheme developed for the translational and rotational dynamics of a rigid body can be applied to \refeqn{ddotx0s} and \refeqn{dotW0s}. Here, we consider a proportional-derivative type nonlinear controller studied in~\cite{GooLeePECC13}. Define position, attitude, and angular velocity tracking error vectors $e_{x_0},e_{R_0},e_{\Omega_0}\in\Re^3$ for the payload as
\begin{align}
e_{x_0} & = x_0 -x_{0_d},\\
e_{R_0} & = \frac{1}{2} (R_{0_d}^T R_0-R_0^T R_{0_d})^\vee,\\
e_{\Omega_0} & = \Omega_0 - R_0^T R_{0_d}\Omega_{0_d}.
\end{align}
The desired resultant control force $F_d\in\Re^3$ and moment $M_d\in\Re^3$ acting on the payload are given in term of these error variables as
\begin{align}
F_d & = m_0 (-k_{x_0}e_{x_0} - k_{\dot x_0} \dot e_{x_0} + \ddot x_{0_d} - g e_3),\label{eqn:Fd}\\
M_d & = -k_{R_0} e_{R_0} - k_{\Omega_0} e_{\Omega_0}  \nonumber\\
& \quad +(R_0^TR_{0_d}\Omega_{0_d})^\wedge J_0 R_0^TR_{0_d} \Omega_{0_d} + J_0 R_0^T R_{0_d} \dot\Omega_{0_d},\label{eqn:Md}
\end{align}
for positive constants $k_{x_0},k_{\dot x_0},k_{R_0},k_{\Omega_0}\in\Re$.

One may try to choose the virtual control input by making the expressions in the right-hand sides of \refeqn{ddotx0s} and \refeqn{dotW0s} identical to $F_d$ and $M_d$, respectively. But, this is not valid as each $\mu_i$ is constrained to be parallel to $q_i$. Instead, we choose the desired value of $\mu_i$, without any constraint, such that 
\begin{align}
\sum_{i=1}^n \mu_{i_d} = F_d,\quad \sum_{i=1}^n \hat\rho_i R_0^T \mu_{i_d} = M_d,\label{eqn:muid0}
\end{align}
or equivalently, using the matrix $\mathcal{P}$ defined at \refeqn{P},
\begin{align*}
\mathcal{P}
\begin{bmatrix}
R_0^T\mu_{1_d} \\ \vdots \\ R_0^T\mu_{n_d}
\end{bmatrix}
=
\begin{bmatrix}
R_0^T F_d \\ M_d
\end{bmatrix}.
\end{align*}
From the assumption stated at \refeqn{A1}, there exists at least one solution to the above matrix equation for any $F_d, M_d$. Here, we find the minimum-norm solution given by
\begin{align}
\begin{bmatrix} \mu_{1_d}\\\vdots\\\mu_{n_d} \end{bmatrix} 
= \mathrm{diag}[R_0,\cdots R_0]\;\mathcal{P}^T (\mathcal{P}\mathcal{P}^T)^{-1}\begin{bmatrix} R_0^T F_d\\M_d\end{bmatrix}.\label{eqn:muid}
\end{align}
The virtual control input $\mu_i$ is selected as the projection of its desired value $\mu_{i_d}$ along $q_i$,
\begin{align}
\mu_i = (\mu_{i_d}\cdot q_i) q_i=q_iq_i^T\mu_{i_d},\label{eqn:mui}
\end{align}
and the desired direction of each link, namely $q_{i_d}\in\Sph^2$ is defined as
\begin{align}
q_{i_d} = -\frac{\mu_{i_d}}{\|\mu_{i_d}\|}.\label{eqn:qid}
\end{align}
It is straightforward to verify that when $q_{i}=q_{i_d}$, the resultant force and moment acting on the payload become identical to their desired values.

Here, the extra degrees of freedom in control inputs are used to minimize the magnitude of the desired tension at \refeqn{muid}, but they can be applied to other tasks, such as controlling the relative configuration of links~\cite{LeeSrePICDC13}. This is referred to future investigation.

\subsection{Design of Normal Components}

Substituting \refeqn{ai0} into \refeqn{dotwi}, the equation of motion for the $i$-link is given by
\begin{align}
\dot\omega_i & = \frac{1}{l_i}\hat q_i a_i-\frac{1}{m_il_i}\hat q_i u_i^\perp.\label{eqn:widotf0}
\end{align}
Here, the normal component of the control input $u_i^\perp$ is chosen such that $q_i\rightarrow q_{i_d}$ as $t\rightarrow\infty$. Control systems for the unit-vectors on the two-sphere have been studied in~\cite{BulLew05,Wu12}. In this paper, we apply a control system developed in terms of the angular velocity in~\cite{Wu12}. For the given desired direction of each link, its desired angular velocity is obtained from the kinematics equation as
\begin{align}
\omega_{i_d} = q_{i_d}\times \dot q_{i_d}.\label{eqn:wid}
\end{align}
Define the direction and the angular velocity tracking error vectors for the $i$-th link, namely $e_{q_i},e_{\omega_i}\in\Re^3$ as 
\begin{align}
e_{q_i} & = q_{i_d}\times q_i,\\
e_{\omega_i} & = \omega_i + \hat q_i^2\omega_{i_d}.
\end{align}
For positive constants $k_{q},k_{\omega}\in\Re$, the normal component of the control input is chosen as
\begin{align}
u_i^\perp & = m_il_i\hat q_i \{-k_q e_{q_i} -k_{\omega}e_{\omega_i} -(q_i\cdot\omega_{i_d})\dot q_i -\hat q_i^2\dot\omega_d\}\nonumber\\
&\quad - m_i\hat q_i^2 a_i.\label{eqn:uiperp}
\end{align}
Note that the expression of $u_i^\perp$ is perpendicular to $q_i$ by definition. Substituting \refeqn{uiperp} into \refeqn{widotf0}, and rearranging by the facts that the matrix $-\hat q_i^2$ corresponds to the orthogonal projection to the plane normal to $q_i$ and $\hat q_i^3=-\hat q_i$, we obtain
\begin{align}
\dot\omega_i & = -k_q e_{q_i} -k_{\omega}e_{\omega_i} -(q_i\cdot\omega_{i_d})\dot q_i -\hat q_i^2\dot\omega_d.\label{eqn:dotwif}
\end{align}

In short, the control force for the simplified dynamic model is given by
\begin{align}
u_i = u_i^\parallel + u_i^\perp.\label{eqn:ui}
\end{align}
The resulting stability properties are summarized as follows.

\begin{prop}\label{prop:SDM}
Consider the simplified dynamic model defined by \refeqn{ddotx0}-\refeqn{dotwi}. For given tracking commands $x_{0_d},R_{0_d}$, a control input is designed as \refeqn{ui}. Then, there exist the values of controller gains, $k_{x_0},k_{\dot x_0},k_{R_0},k_{\Omega_0}, k_{q},k_{\omega}$ such that the zero equilibrium of tracking errors $(e_{x_0},\dot e_{x_0},e_{R_0},e_{\Omega_0},e_{q_i},e_{\omega_i})$ is exponentially stable. 
\end{prop}
\begin{proof}
See Appendix \ref{sec:prfSDM}
\end{proof}

\begin{remark}
At \refeqn{qid}, the negative sign appeared to make the tension at each cable positive when $q_i=q_{i_d}$.  Assuming that the tracking errors $e_{x_0},\dot e_{x_0},e_{R_0},e_{\Omega_0}$ and the variables $\ddot x_{0_d},\Omega_{0_d},\dot\Omega_{0_d}$ obtained from the desired trajectories are sufficiently small, this guarantees that quadrotors remain above the payload. If desired, the negative sign at \refeqn{qid} can be eliminated to place quadrotors below the payload, resulting in a tracking control of an inverted rigid body multi-link pendulum, that can be considered as a generalization of a flying spherical  inverted spherical pendulum~\cite{LeeSrePICDC13}. 
\end{remark}

\section{Control System Design for Full Dynamic Model}\label{sec:FDM}

The control system designed at the previous section is based on a simplifying assumption that each quadrotor can generates a thrust along any direction. However, the dynamics of quadrotor is underactuated since the direction of the total thrust is always parallel to its third body-fixed axis, while the magnitude of the total thrust can be arbitrarily changed. This can be directly observed from the expression of the total thrust, $u_i = -f_i R_i e_3$, where $f_i$ is the total thrust magnitude, and $R_ie_3$ corresponds to the direction of the third body-fixed axis. Whereas, the rotational attitude dynamics is fully actuated by the arbitrary control moment $M_i$.

Based on these observations, the attitude of each quadrotor is controlled such that the third body-fixed axis becomes parallel to the direction of the ideal control force $u_i$ designed in the previous section. The desired direction of the third body-fixed axis of the $i$-th quadrotor, namely $b_{3_i}\in\Sph^2$ is given by
\begin{align}
b_{3_i} = -\frac{u_i}{\|u_i\|}.\label{eqn:b3i}
\end{align}
This provides two-dimensional constraint on the three-dimensional desired attitude of each quadrotor, and there remains one degree of freedom. To resolve it, the desired direction of the first body-fixed axis $b_{1_i}(t)\in\Sph^2$ is introduced as a smooth function of time. Due to the fact that the first body-fixed axis is normal to the third body-fixed axis, it is impossible to follow an arbitrary command $b_{1_i}(t)$ exactly. Instead, its projection onto the plane normal to $b_{3_i}$ is followed, and the desired direction of the second body-fixed axis is chosen to constitute an orthonormal frame~\cite{LeeLeoPICDC10}. More explicitly, the desired attitude of the $i$-th quadrotor is given by
\begin{align}
R_{i_c} = \begin{bmatrix}-\frac{(\hat b_{3_i})^2 b_{1_i}}{\|(\hat b_{3_i})^2 b_{1_i}\|}, &
\frac{\hat b_{3_i}b_{1_i}}{\|\hat b_{3_i}b_{1_i}\|},& b_{3_i}\end{bmatrix},
\end{align}
which is guaranteed to be an element of $\SO$. The desired angular velocity is obtained from the attitude kinematics equation, $\Omega_{i_c} = (R_{i_c}^T\dot R_{i_c})^\vee\in\Re^3$.  

Define the tracking error vectors for the attitude and the angular velocity of the $i$-th quadrotor as
\begin{align}
e_{R_i} = \frac{1}{2}(R_{i_c}^T R_i -R_i^T R_{i_c})^\vee,\quad
e_{\Omega_i} = \Omega_i - R_i^T R_{i_c}\Omega_{i_c}.
\end{align}
The thrust magnitude is chosen as the length of $u_i$, projected on to $-R_ie_3$, and the control moment is chosen as a tracking controller on $\SO$:
\begin{align}
f_i & = - u_i\cdot R_i e_3,\label{eqn:fi}\\
M_i & = -\frac{k_R}{\epsilon^2} e_{R_i} -\frac{k_\Omega}{\epsilon} e_{\Omega_i} + \Omega_i\times J_i\Omega_i\nonumber\\
&\quad - J_i (\hat\Omega_i R_i^T R_{i_c}\Omega_{i_c}-R_i^T R_{i_c}\dot\Omega_{i_c}),\label{eqn:Mi}
\end{align}
where $\epsilon,k_R,k_\Omega$ are positive constants. 

Stability of the corresponding controlled systems for the full dynamic model can be studied by showing the the error due to the discrepancy between the desired direction $b_{3_i}$ and the actual direction $R_i e_3$ can be compensated via Lyapunov analysis~\cite{LeeLeoPICDC10}, or singular perturbation theory can be applied to the attitude dynamics of quadrotors~\cite{SreLeePICDC13,LeeSrePICDC13}. For both cases, the structures of the control systems are identical, and here we use singular perturbation for simplicity.

\begin{prop}\label{prop:FDM}
Consider the full dynamic model defined by \refeqn{ddotx0}-\refeqn{dotWi}. For given tracking commands $x_{0_d},R_{0_d}$ and the desired direction of the first body-fixed axis $b_{1_i}$, control inputs for quadrotors are designed as \refeqn{fi} and \refeqn{Mi}. Then, there exists $\epsilon^\star>0$, such that for all $\epsilon<\epsilon^\star$, the zero equilibrium of the tracking errors $(e_{x_0}, \dot e_{x_0}, e_{R_0}, e_{\Omega_0},e_{q_i},e_{\omega_i}, e_{R_i},e_{\Omega_i}$ is exponentially stable. 
\end{prop}
\begin{proof}
See Appendix \ref{sec:prfFDM}.
\end{proof}

\begin{figure}
\centerline{
	\subfigure[3D perspective]{
		\includegraphics[width=0.95\columnwidth]{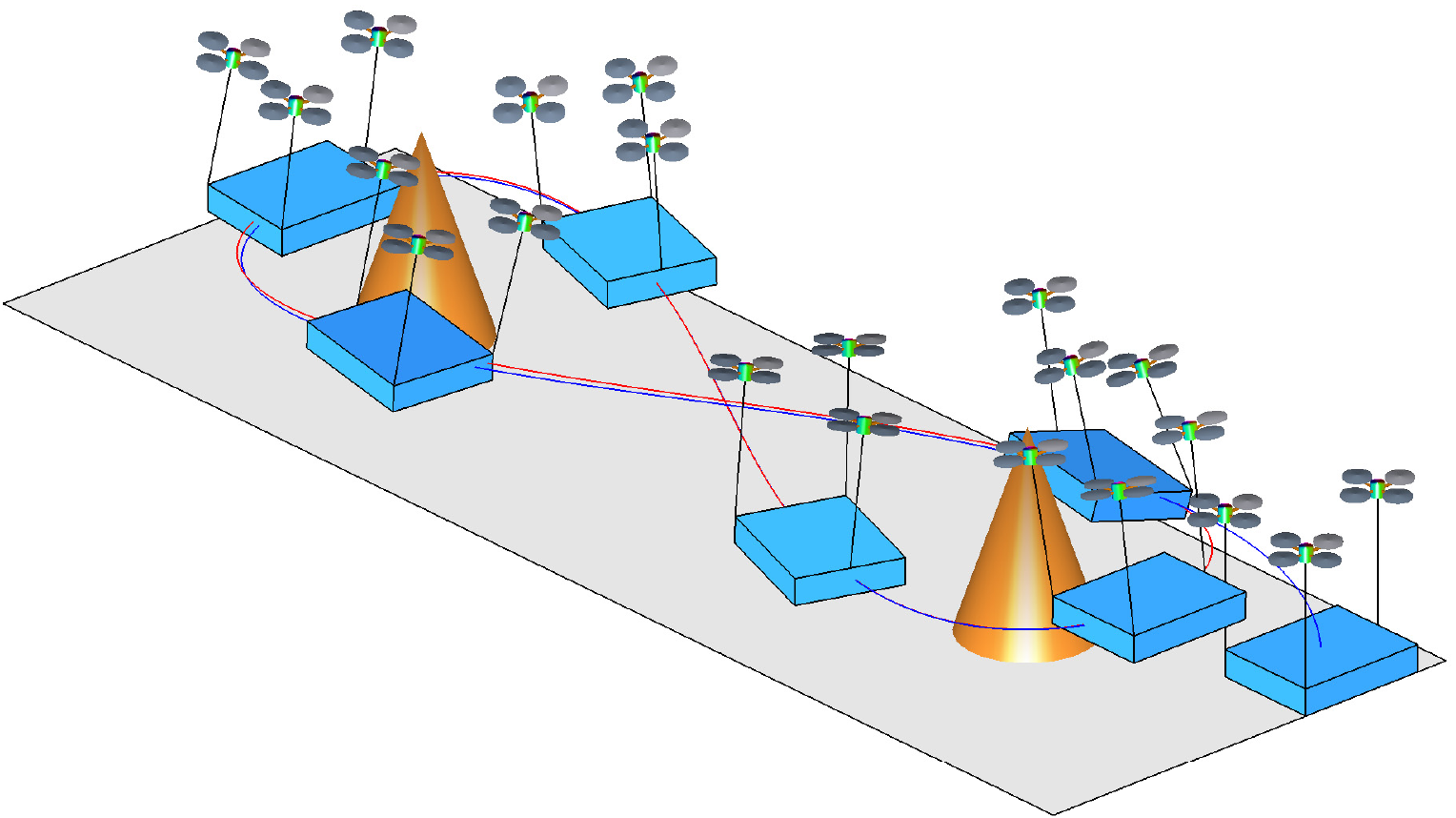}}
}
\centerline{
	\subfigure[Top view]{
	\setlength{\unitlength}{0.1\columnwidth}\scriptsize
	\begin{picture}(9.5,3.2)(0,0)
	\put(0,0){\includegraphics[width=0.95\columnwidth]{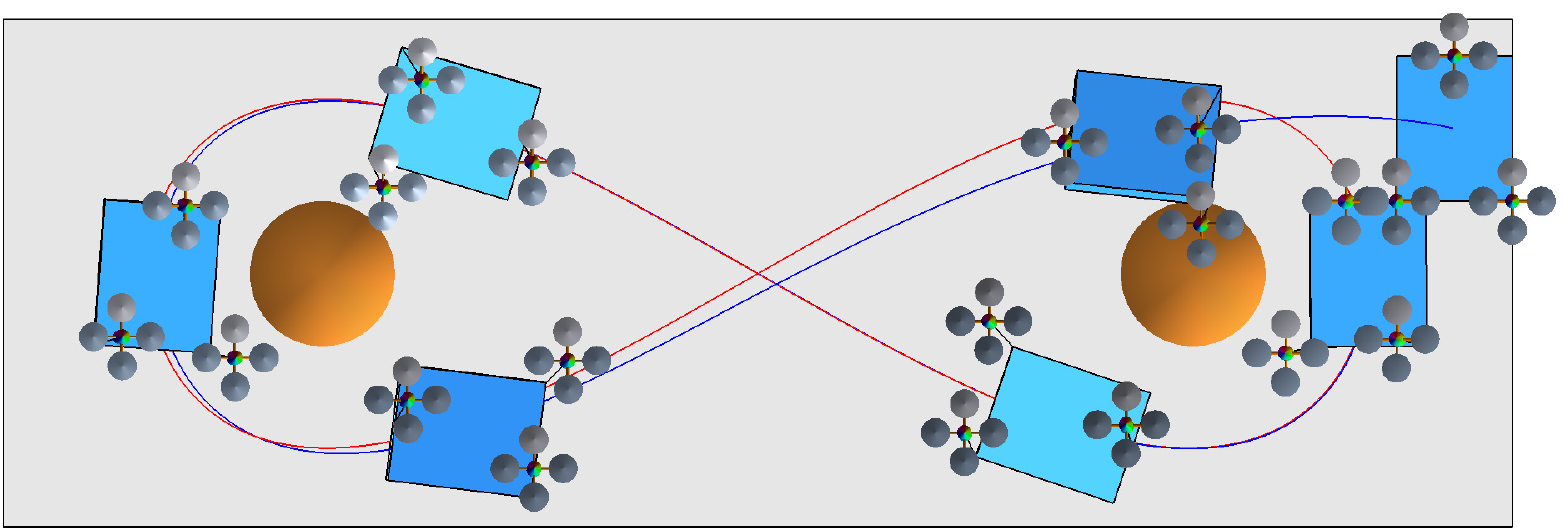}}
	\put(8.4,3.2){$t=0$}
	\put(6.4,2.8){$t=1.66$}
	\put(2.25,-0.15){$t=3.33$}
	\put(0.7,2.3){$t=5$}
	\put(2.25,3.0){$t=6.66$}
	\put(5.8,-0.15){$t=8.33$}
	\put(7.9,0.6){$t=10$}
	\end{picture}}	
}
\centerline{
	\subfigure[Side view]{
		\includegraphics[width=0.95\columnwidth]{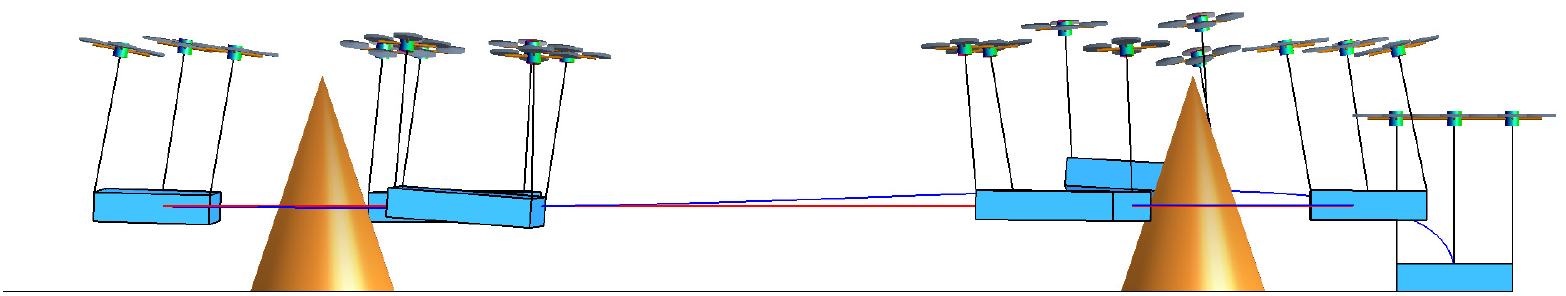}}
}
\caption{Snapshots of controlled maneuver (red:desired trajectory, blue:actual trajectory). A short animation illustrating this maneuver is available at {\href{http://youtu.be/h65IaQx_w6c}{http://youtu.be/h65IaQx\_w6c}}.}\label{fig:SS}
\end{figure}
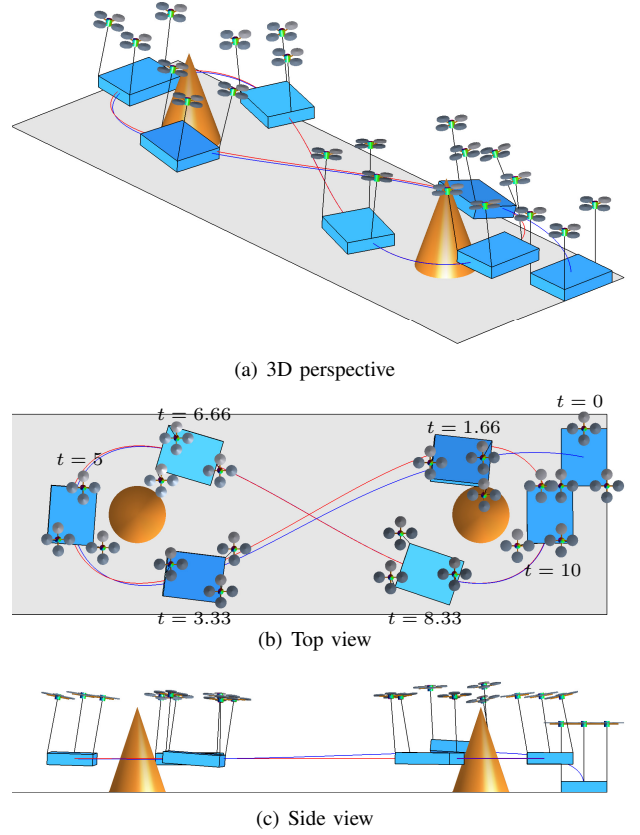

\section{Numerical Example}\label{sec:NE}
We consider a numerical example where three quadrotors ($n=3$) transport a rectangular box along a figure-eight curve, that is a special case of Lissajous figure shaped like an $\infty$ symbol.

More explicitly, the mass of the payload is $m_0=1.5\,\mathrm{kg}$, and its length, width, and height are $1.0\,\mathrm{m}$, $0.8\,\mathrm{m}$, and $0.2\,\mathrm{m}$, respectively. Mass properties of three quadrotors are identical, and they are given by
\begin{gather*}
m_i=0.755\,\mathrm{kg},\quad J_i=\mathrm{diag}[0.0820,\, 0.0845,\, 0.1377]\,\mathrm{kgm^2}.
\end{gather*}
The length of cable is $l_i=1\,\mathrm{m}$, and they are attached to the following points of the payload.
\begin{gather*}
\rho_1 = [0.5,\,0,\,-0.1]^T,\\
\rho_1 = [-0.5,\,0.4,\,-0.1]^T,\quad
\rho_3 = [-0.5,\,-0.4,\,-0.1]^T.
\end{gather*}
In other words, the first link is attached to the center of the top, front edge, and the remaining two links are attached to the vertices of the top, rear edge (see Figure \ref{fig:DM}). 

The desired trajectory of the payload is chosen as
\begin{align*}
x_{0_d}(t) = [1.2\sin (0.4\pi t),\, 4.2\cos (0.2\pi t),\, -0.5]^T\,\mathrm{m}.
\end{align*}
The desired attitude of the payload is chosen such that its first axis is tangent to the desired path, and the third axis is parallel to the direction of gravity, it is given by
\begin{align*}
R_{0_d}(t) = \begin{bmatrix}\frac{\dot x_{0_d}}{\|\dot x_{0_d}\|}&\frac{\hat e_3\dot x_{0_d}}{\|\hat e_3\dot x_{0_d}\|}& e_3\end{bmatrix}.
\end{align*}
Initial conditions are chosen as
\begin{gather*}
x_0(0)=[1,\,4.8,\,0]^T,\quad v_0(0)=0_{3\times 1},\\
q_i(0)= e_3,\; \omega_i(0)=0_{3\times 1},\;
R_i(0)=I_{3\times 3},\; \Omega_i(0)=0_{3\times 1}.
\end{gather*}

\begin{figure}
\centerline{
	\subfigure[Position of payload ($x_0$:blue, $x_{0_d}$:red)]{
		\includegraphics[width=0.48\columnwidth]{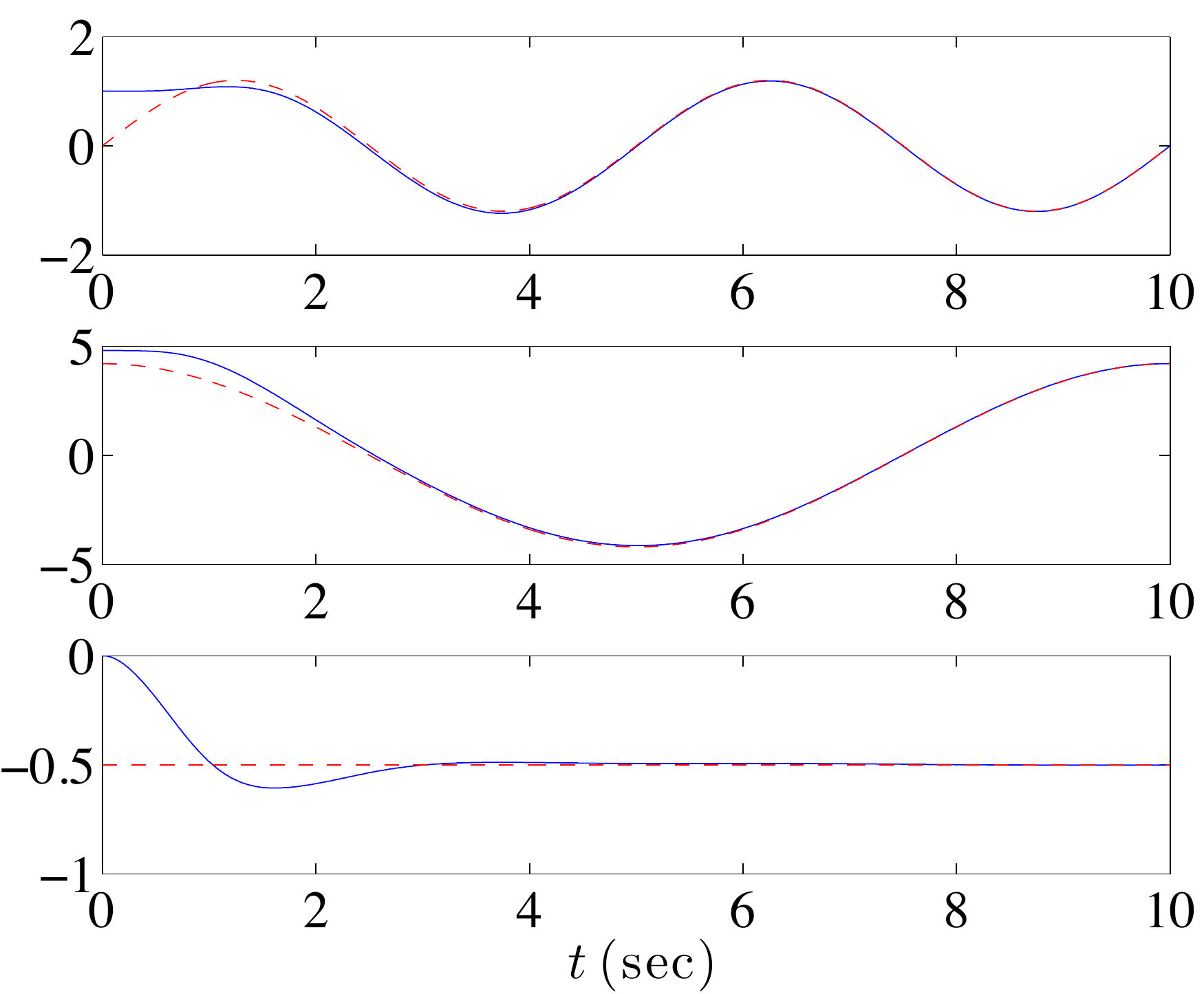}}
	\hfill
	\subfigure[Attitude tracking error of payload $\Psi_0=\frac{1}{2}\|R_0-R_{0_d}\|^2$]{
		\includegraphics[width=0.48\columnwidth]{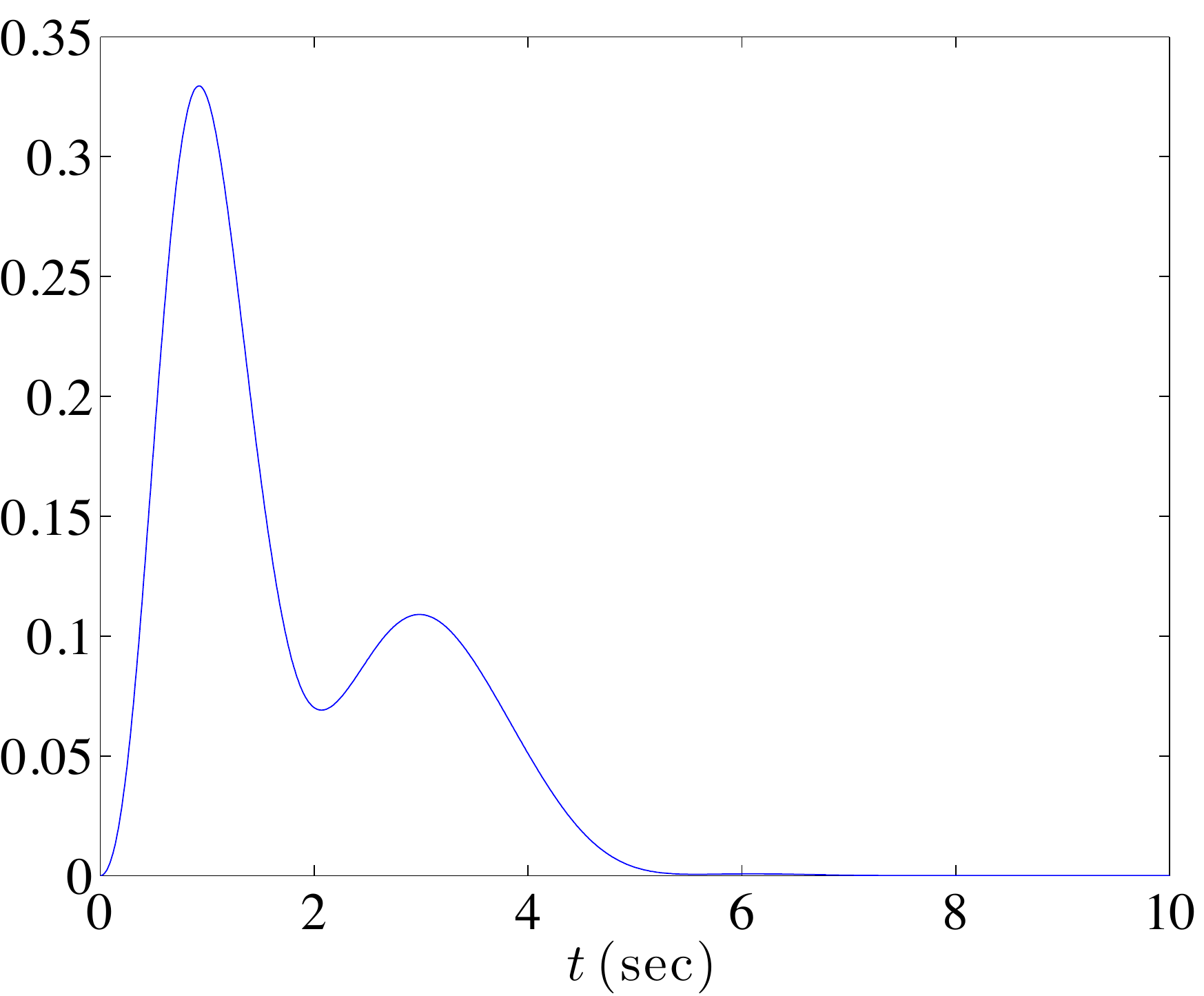}}
}
\centerline{
	\subfigure[Link direction error $\Psi_{q_i}=\frac{1}{2}\|q_i-q_{i_d}\|^2$]{
		\includegraphics[width=0.48\columnwidth]{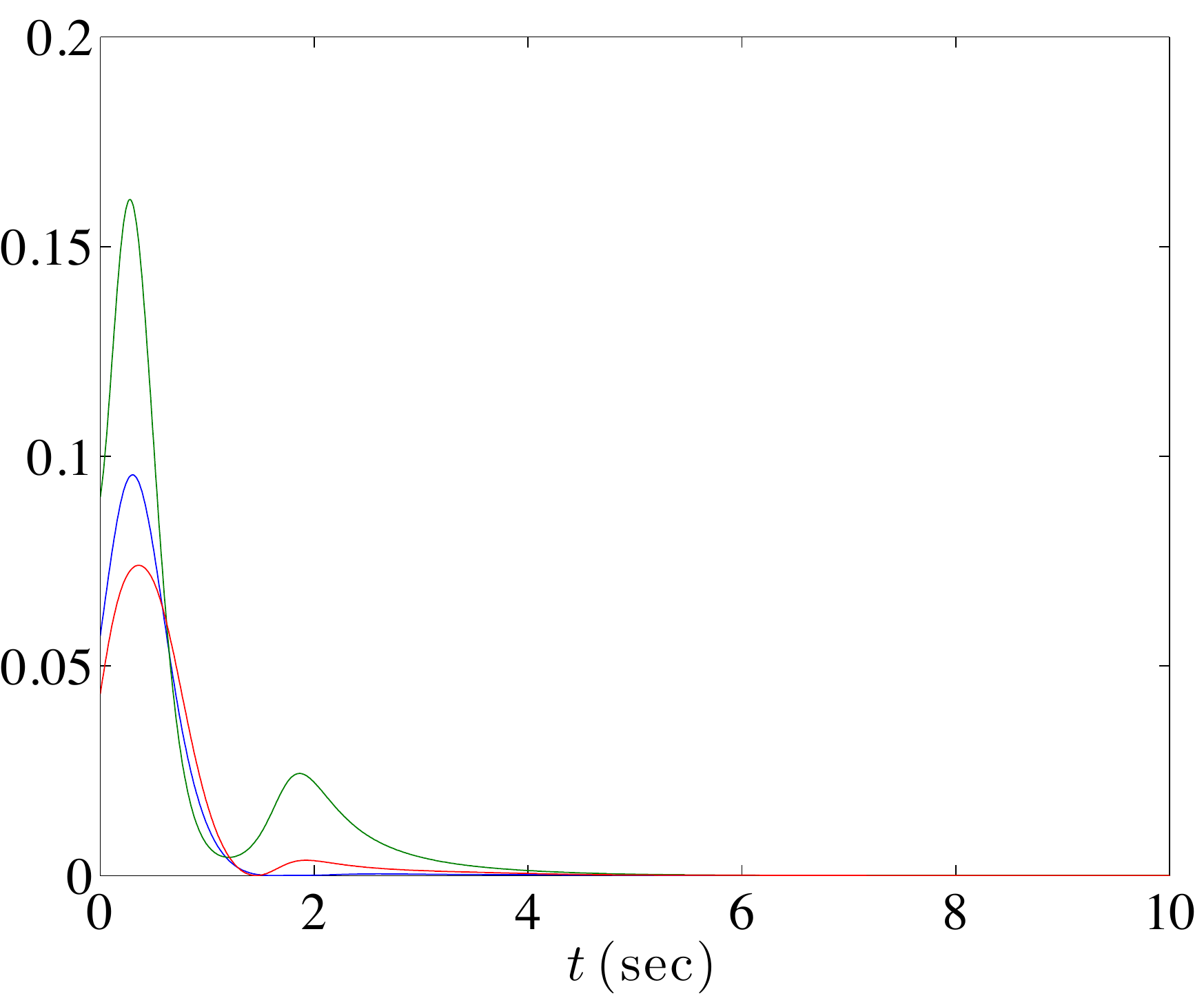}}
	\hfill
	\subfigure[Attitude tracking error of quadrotors $\Psi_i=\frac{1}{2}\|R_i-R_{i_d}\|^2$]{
		\includegraphics[width=0.48\columnwidth]{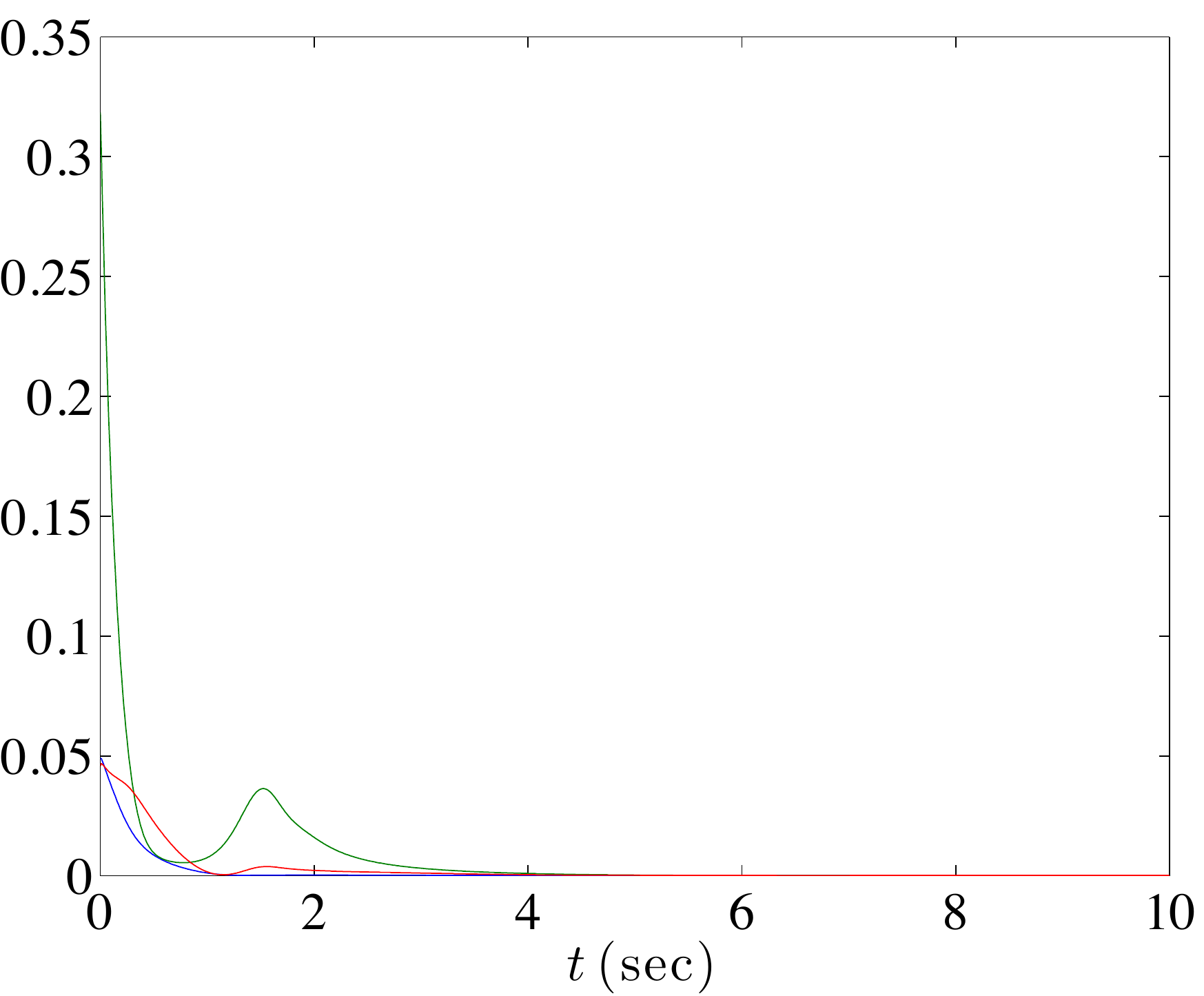}}
}
\centerline{
	\subfigure[Tension at links]{
		\includegraphics[width=0.45\columnwidth]{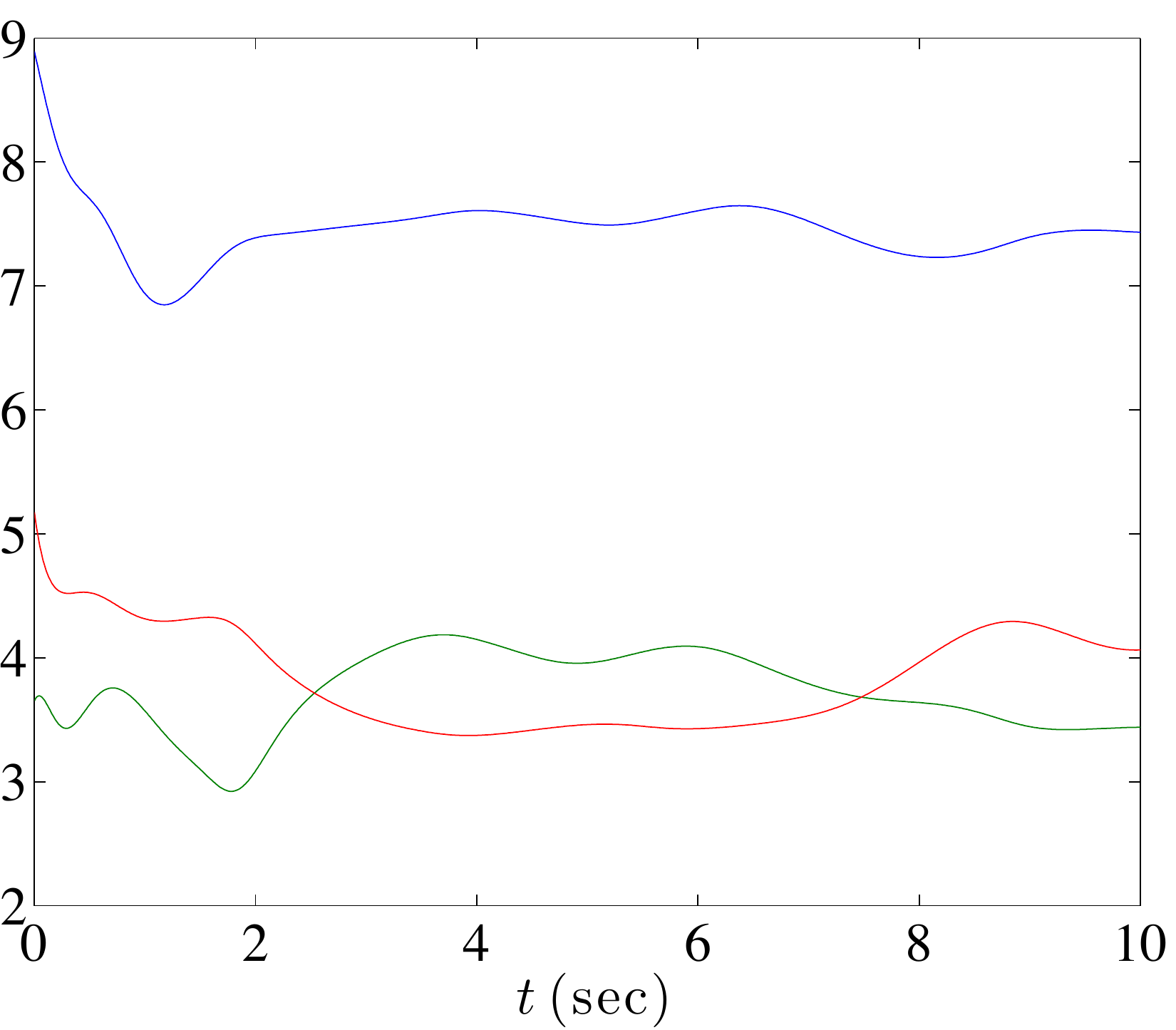}}
	\hfill
	\subfigure[Control input for quadrotors $f_i,M_i$]{
		\includegraphics[width=0.48\columnwidth]{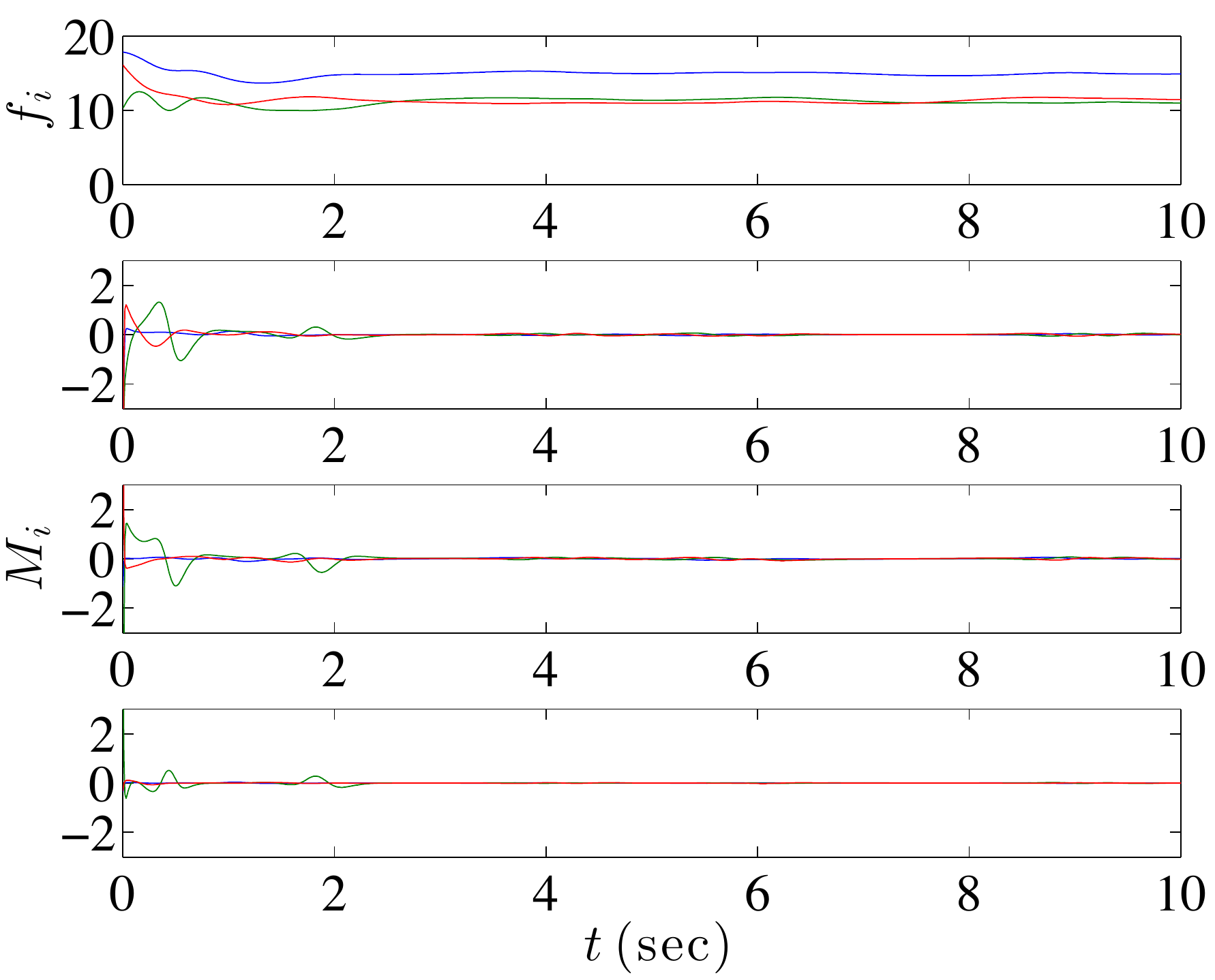}}
}
\caption{Simulation results for tracking errors and control inputs. (for figures (c)-(f): $i=1$:blue, $i=2$:green, $i=3$:red)}\label{fig:SR}
\end{figure}

The corresponding simulation results are presented at Figures \ref{fig:SS} and \ref{fig:SR}. Figure \ref{fig:SS} illustrates the desired trajectory that is shaped like a figure-eight curve around two obstacles represented by cones, and the actual maneuver of the payload and quadrotors. Figure \ref{fig:SR} shows tracking errors for the position and the attitude of the payload, tracking errors for the link directions and the attitude of quadrotors, as well as tension and control inputs. These illustrate excellent tracking performances of the proposed control system. 

\appendix

\subsection{Lagrangian Mechanics}\label{sec:EOM}

\paragraph{Derivatives of Lagrangian}

Here, we develop the equations of motion for the Lagrangian given by \refeqn{TT} and \refeqn{UU}. The derivatives of the Lagrangian are given by
\begin{align}
\D_{\dot x_0} \mathcal{L} & = m_T \dot x_0 +\sum_{i=1}^n m_i(R_0\hat\Omega_0\rho_i - l_i\dot q_i),\label{eqn:Ddotx0L}\\
\D_{\dot q_i} \mathcal{L} & = \sum_{i=1}^n m_i(l_i^2\dot q_i -l_i \dot x_0 -l_i R_0\hat\Omega_0\rho_i),\\
\D_{\Omega_0} \mathcal{L} & = \bar J_0\Omega_0 +\sum_{i=1}^n m_i\hat\rho_i R_0^T(\dot x_0-l_i\dot q_i),\\
\D_{\Omega_i} \mathcal{L} & = J_i\Omega_i,\\
\D_{x_0}\mathcal{L} & = m_T g e_3,\\
\D_{q_i}\mathcal{L} & = -  m_il_ige_3,\label{eqn:DqiL}
\end{align}
where $\bar J_0 = J_0 -\sum_{i=1}^n m_i\hat\rho_i^2$. The variation of a rotation matrix is represented by $\delta R_j=R_j\hat\eta_j$ for $\eta_j\in\Re^3$~\cite{Lee08}. Using this the derivative of the Lagrangian with respect to $R_j$ can be written as
\begin{align}
\D_{R_0}&\mathcal{L}\cdot\delta R_0 = \sum_{i=1}^n m_i R_0\hat\eta_0 \hat\Omega_0\rho_i\cdot(\dot x_i -l_i \dot q_i) + m_ige_3\cdot R_0\hat\eta_0\rho_i\nonumber\\
& = \sum_{i=1}^n m_i\{\widehat{\hat\Omega_0\rho_i}R_0^T(\dot x_0-l_i \dot q_i)+g\hat\rho_i R_0^T e_3\}\cdot \eta_0\nonumber\\
& \triangleq \d_{R_0}\mathcal{L}\cdot\eta_0,\label{eqn:dR0L}
\end{align}
where $\d_{R_0}\mathcal{L}\in(\Re^3)^*\simeq \Re^3$ is referred to as left-trivialized derivatives. Substituting $\delta R_j=R_j\hat\eta_j$ into the attitude kinematic equations \refeqn{dotRi} and rearranging, the variation of the angular velocity can be written as $\delta\Omega_j=\dot\eta_j + \Omega_j\times\eta_j$. For the variation model of $q_i$ given at \refeqn{delqi}, we have $\delta q_i =\xi_i\times q_i$ and $\dot \xi_i = \dot\xi_i\times q_i + \xi_i\times \dot q_i$. 

\paragraph{Lagrange-d'Alembert Principle}
Let $\mathfrak{G}=\int_{t_0}^{t_f}\mathcal{L}\,dt$ be the action integral. Using the above equations, the infinitesimal variation of the action integral can be written as
\begin{align*}
\delta \mathfrak{G} & = \int_{t_0}^{t_f} \D_{\dot x_0} \mathcal{L}\cdot\delta\dot x_0 + \D_{ x_0} \mathcal{L}\cdot\delta x_0\\
&\quad + \D_{\Omega_0} \mathcal{L}\cdot(\dot\eta_0+\Omega_0\times\eta_0)+ \d_{ R_0} \mathcal{L}\cdot \eta_0\\
&\quad +\sum_{i=1}^n \D_{\dot q_i} \mathcal{L}\cdot(\dot\xi_i\times q_i+\xi_i\times\dot q_i) + \D_{ q_i} \mathcal{L}\cdot(\xi_i\times q_i)\\
&\quad + \sum_{i=1}^n \D_{\Omega_i} \mathcal{L}\cdot(\dot\eta_i+\Omega_i\times\eta_i).
\end{align*}
The total thrust at the $i$-th quadrotor with respect to the inertial frame is denoted by $u_i = -f_iR_ie_3\in\Re^3$ and the total moment at the $i$-th quadrotor is defined as $M_i\in\Re^3$. The corresponding virtual work can be written as
\begin{align*}
\delta\mathcal{W} & = \int_{t_0}^{t_f} \sum_{i=1}^n u_i\cdot\braces{\delta x_0 + R_0\hat\eta_0\rho_i -l_i\xi_i\times q_i}+ M_i\cdot\eta_i.
\end{align*}
According to Lagrange-d'Alembert principle, we have $\delta\mathfrak{G}=-\delta\mathcal{W}$ for any variation of trajectories with fixed end points.

By using integration by parts and rearranging, we obtain the following Euler-Lagrange equations:
\begin{gather*}
\frac{d}{dt}\D_{\dot x_0}\mathcal{L}-\D_{x_0}\mathcal{L} = \sum_{i=1}^n u_i,\\
\frac{d}{dt}\D_{\Omega_0}\mathcal{L}+\Omega_0\times \D_{\Omega_0}\mathcal{L}
-\d_{R_0}\mathcal{L} = \sum_{i=1}^n\hat\rho_0 R_0^T u_i,\\
\hat q_i \frac{d}{dt}\D_{\dot q_i}\mathcal{L}-\hat q_i \D_{q_i}\mathcal{L} =  -l_i\hat q_i u_i,\\
\frac{d}{dt}\D_{\Omega_i}\mathcal{L}+\Omega_i\times \D_{\Omega_i}\mathcal{L} = M_i.
\end{gather*}
Substituting \refeqn{Ddotx0L}-\refeqn{dR0L} into these, and rearranging by the fact that $\ddot q_i = -\hat q_i \dot\omega_i -\|\omega_i\|^2 q_i$ and $\hat q_i\ddot q_i = -\hat q_i^2 \dot\omega_i=\dot\omega_i$~\cite{LeeLeoIJNME08}, the equations of motion are given by
\begin{gather}
m_T\ddot x_0 + \sum_{i=1}^n m_i(-R_0\hat\rho_i\dot\Omega_0 + l_i\hat q_i\dot\omega_i) +\sum_{i=1}^n m_iR_0\hat\Omega_0^2\rho_i\nonumber\\
 + m_il_i\|\omega_i\|^2q_i = m_T ge_3+\sum_{i=1}^n u_i,\label{eqn:x0ddot0}\\
\bar J_0\dot\Omega_0 + \sum_{i=1}^n m_i\hat\rho_i R_0^T(\ddot x_0 +l_i\hat q_i\dot\omega_i +l_i\|\omega_i\|^2 q_i)  +\hat\Omega_0 \bar J_0\Omega_0\nonumber\\
=\sum_{i=1}^n \hat\rho_i R_0^T (u_i+m_i ge_3),\label{eqn:W0dot0}\\
m_il_i\dot\omega_i -m_i\hat q_i\ddot x_0 +m_i \hat q_i R_0\hat\rho_i\dot\Omega_0
-m_i\hat q_iR_0\hat\Omega_0^2\rho_i\nonumber\\
 = -\hat q_i(u_i+m_ig e_3),\label{eqn:widot0}\\
J_i\dot\Omega_i + \Omega_i\times J_i\Omega_i =M_i,
\end{gather}
where $m_T = m_0+\sum_{i=1}^n m_i\in\Re^3$ and $\bar J_0 = J_0 -\sum_{i=1}^n m_i\hat\rho_i^2\in\Re^{3\times 3}$. This can be rewritten in a matrix form as given at \refeqn{EOMM}.
\begin{figure*}
{\tiny\selectfont
\begin{gather}
\begin{bmatrix}
m_T & \sum_{i=1}^n -m_i R_0\hat\rho_i & m_1l_1\hat q_1 & \cdots & m_nl_n\hat q_n\\
\sum_{i=1}^n m_i\hat\rho_iR_0^T & \bar J_0 & m_1l_1\hat\rho_1 R_0^T\hat q_1 & \cdots & m_nl_n\hat\rho_n R_0^T\hat q_n\\
-m_1l_1\hat q_1 & m_1l_1\hat q_1 R_0\hat\rho_1 & m_1l_1^2 & \cdots & 0\\
\vdots & \vdots & \vdots & &\vdots\\
-m_nl_n\hat q_n & m_nl_n\hat q_n R_0\hat\rho_n & 0 & \cdots & m_nl_n^2\\
\end{bmatrix}
\begin{bmatrix} \ddot x_0 \\ \dot\Omega_0 \\ \dot\omega_1 \\ \vdots \\ \dot\omega_n \end{bmatrix}
=\begin{bmatrix}
-\sum_{i=1}^n \{m_iR_0\hat\Omega_0^2\rho_i + m_il_i\|\omega_i\|^2q_i\}+ m_T ge_3+\sum_{i=1}^n u_i\\
-\hat\Omega_0\bar J_0\Omega_0-\sum_{i=1}^n m_il_i\hat\rho_i R_0^T\|\omega_i\|^2 q_i  +
\sum_{i=1}^n \hat\rho_i R_0^T (u_i+m_i ge_3)\\
m_1l_1\hat q_1 R_0\hat\Omega_0^2\rho_1 - l_1\hat q_1(u_1+m_1g e_3)\\
\vdots\\
m_nl_n\hat q_n R_0\hat\Omega_0^2\rho_n - l_n\hat q_i(u_n+m_ng e_3)
\end{bmatrix}.\label{eqn:EOMM}
\end{gather}}
\end{figure*}

Next, we substitute \refeqn{widot0} into \refeqn{x0ddot0} and \refeqn{W0dot0} to eliminate the dependency of $\dot\omega_i$ in the expressions for $\ddot x_0$ and $\dot\Omega_0$. Using the fact that $I+\hat q_i^2 = q_iq_i^T$ for any $q_i\in\Sph^2$ and $\hat\Omega_0\hat\rho_i\Omega_0=-\hat\rho_i\hat\Omega_0^2\rho_i$ for any $\Omega_0,\rho_i\in\Re^3$, we obtain \refeqn{ddotx0} and \refeqn{dotW0} after rearrangements and simplifications. It is straightforward to see that \refeqn{widot0} is equivalent to \refeqn{dotwi}.

\subsection{Proof of Proposition \ref{prop:SDM}}\label{sec:prfSDM}

\paragraph{Error Dynamics}

From \refeqn{ddotx0s} and \refeqn{mui}, the dynamics of the position tracking error is given by
\begin{align*}
m_0\ddot e_{x_0} & = m_0(ge_3-\ddot x_{0_d}) + \sum_{i=1}^n q_iq_i^T \mu_{i_d}.
\end{align*}
From \refeqn{muid0} and \refeqn{Fd}, this can be rearranged as
\begin{align}
\ddot e_{x_0} & = ge_3-\ddot x_{0_d} + \frac{1}{m_0}F_d  +Y_x,\nonumber\\
& = -k_{x_0}e_{x_0} - k_{\dot x_0} \dot e_{x_0}  +Y_x,
\label{eqn:ddotex0}
\end{align}
where the last term $Y_x\in\Re^3$ to the error caused by the difference between $q_i$ and $q_{i_d}$, and it is given by
\begin{align*}
Y_x=\frac{1}{m_0}\sum_{i=1}^n (q_iq_i^T -I) \mu_{i_d}.
\end{align*}
We have $\mu_{i_d}=q_{i_d}q_{i_d}^T\mu_{i_d}$ from \refeqn{qid}. Using this, the error term can be written in terms of $e_{q_i}$ as
\begin{align}
Y_x  & = \frac{1}{m_0}\sum_{i=1}^n 
 (q_{i_d}^T\mu_{i_d})\{(q_i^Tq_{i_d})q_i-q_{i_d}\}\nonumber\\
& = -\frac{1}{m_0}\sum_{i=1}^n (q_{i_d}^T\mu_{i_d}) \hat q_i e_{q_i}.\label{eqn:Yx}
\end{align}
Using \refeqn{muid}, an upper bound of $Y_x$ can be obtained as
\begin{align*}
\|Y_x\| & \leq \frac{1}{m_0} \sum_{i=1}^n \|\mu_{i_d}\|\|e_{q_i}\|\leq  \sum_{i=1}^n\gamma (\|F_d\|+\|M_d\|)\|e_{q_i}\|,
\end{align*}
where $\gamma= \frac{1}{m_0\sqrt{\lambda_{m}[\mathcal{P}\mathcal{P}^T]}}$. From \refeqn{Fd} and \refeqn{Md}, this can be further bounded by
\begin{align}
\|Y_x\|  \leq & \sum_{i=1}^n \{\beta(k_{x_0}\|e_{x_0}\|+k_{\dot x_0}\|\dot e_{x_0}\|)\nonumber\\
&\quad  + \gamma(k_{R_0} \|e_{R_0}\| + k_{\Omega_0}\|e_{\Omega_0}\|) + B\}\|e_{q_i}\|,\label{eqn:YxB}
\end{align}
for some positive constant $B$ that is determined by the given desired trajectories of the payload, and $\beta=m_0\gamma$. Throughout the remaining parts of the proof, any bound that can be obtained from $x_{0_d},R_{0_d}$ is denoted by $B$ for simplicity. In short, the position tracking error dynamics of the payload can be written as \refeqn{ddotex0}, where the error term is bounded by \refeqn{YxB}.

Similarly, we find the attitude tracking error dynamics for the payload as follows. Using \refeqn{dotW0s}, \refeqn{Md}, and \refeqn{mui}, the time-derivative of $J_0 e_{\Omega_0}$ can be written as
\begin{align}
J_0\dot e_{\Omega_0} & = (J_0e_{\Omega_0}+d)^\wedge e_{\Omega_0} - k_{R_0}e_{R_0}-k_{\Omega_0}e_{\Omega_0}+Y_R,\label{eqn:doteW0}
\end{align}
where $d=(2J_0-\trs{J_0}I)R_0^T R_{0_d}\Omega_{0_d}\in\Re^3$~\cite{GooLeePECC13}. Note that the term $d$ is bounded. The error term in the attitude dynamics of the payload, namely $Y_R\in\Re^3$ is given by
\begin{align}
Y_R & = \sum_{i=1}^n\hat\rho_i R_0^T (q_iq_i^T - I)\mu_{i_d}
= -\sum_{i=1}^n\hat\rho_i R_0^T (q_{i_d}^T \mu_{i_d}) \hat q_i e_{q_i}.\label{eqn:YR}
\end{align}
Similar with \refeqn{YxB}, an upper bound of $Y_R$ can be obtained as
\begin{align}
\|Y_R\|  \leq & \sum_{i=1}^n\{\delta_i (k_{x_0}\|e_{x_0}\|+k_{\dot x_0}\|\dot e_{x_0}\|)\nonumber\\
&\quad  + \sigma_i(k_{R_0} \|e_{R_0}\| + k_{\Omega_0}\|e_{\Omega_0}\|) + B\}\|e_{q_i}\|,\label{eqn:YRB}
\end{align}
where $\delta_i = m_0\frac{\|\hat\rho_i\|}{\sqrt{\lambda_{m}[\mathcal{P}\mathcal{P}^T]}},\sigma_i=\frac{\delta_i}{m_0}\in\Re$.

Next, from \refeqn{dotwif}, the time-derivative of the angular velocity error, projected on to the plane normal to $q_i$ is given as
\begin{align}
-\hat q_i^2 \dot e_{\omega_i} 
& = -k_q e_{q_i} - k_\omega e_{\omega_i}.\label{eqn:dotewi}
\end{align}

\paragraph{Stability Proof}

Define an attitude configuration error function $\Psi_{R_0}$ for the payload as
\begin{align*}
\Psi_{R_0} = \frac{1}{2}\trs{I-R_{0_d}^T R_0}.
\end{align*}
It is positive-definite about $R_0=R_{0_d}$, and $\dot\Psi_{R_0} = e_{R_0}\cdot e_{\Omega_0}$~\cite{LeeLeoPICDC10,GooLeePECC13}. We also introduce a configuration error function $\Psi_{q_i}$ for each link that is positive-definite about $q_i=q_{i_d}$ as
\begin{align*}
\Psi_{q_i} = 1-q_{i}\cdot q_{i_d}.
\end{align*}
For positive constants $e_{x_{\max}}, \psi_{R_0}, \psi_{q_i}\in\Re$, consider the following open domain containing the zero equilibrium of tracking error variables:
\begin{align}
D=\{&(e_{x_0},\dot e_{x_0}, e_{R_0}, e_{\Omega_0}, e_{q_i}, e_{\omega_i})\in(\Re^3)^4\times (\Re^3\times\Re^3)^n\,|\,\nonumber\\
& \|e_{x_0}\|< e_{x_{\max}},\,\Psi_{R_0}< \psi_{R_0}<1,\,\Psi_{q_i}< \psi_{q_i}<1\}.\label{eqn:D}
\end{align}
In this domain, we have $\|e_{R_0}\|=\sqrt{\Psi_{R_0}(2-\Psi_{R_0}} \leq \sqrt{\psi_{R_0}(2-\psi_{R_0})} \triangleq \alpha_0 < 1$, and $\|e_{q_i}\|=\sqrt{\Psi_{q_i}(2-\Psi_{q_i})} \leq \sqrt{\psi_{q_i}(2-\psi_{q_i})} \triangleq \alpha_i < 1$. It is assumed that $\psi_{q_i}$ is sufficiently small such that $n\alpha_i\beta < 1$. 

We can show that the configuration error functions are quadratic with respect to the error vectors in the sense that
\begin{gather*}
\frac{1}{2}\|e_{R_0}\|^2 \leq \Psi_{R_0} \leq \frac{1}{2-\psi_{R_0}} \|e_{R_0}\|^2,\\
\frac{1}{2}\|e_{q_i}\|^2 \leq \Psi_{q_i} \leq \frac{1}{2-\psi_{q_i}} \|e_{q_i}\|^2,
\end{gather*}
where the upper bounds are satisfied only in the domain $D$. 

Define a Lyapunov function as
\begin{align*}
\mathcal{V} & = \frac{1}{2}\|\dot e_{x_0}\|^2 + \frac{1}{2}k_{x_0}\|e_{x_0}\|^2   +c_x e_{x_0}\cdot \dot e_{x_0}\\
&\quad + \frac{1}{2}e_{\Omega_0}\cdot J_0\Omega_0 + k_{R_0}\Psi_{R_0} + c_R e_{R_0}\cdot J_0e_{\Omega_0}\\
&\quad+ \sum_{i=1}^n \frac{1}{2}\|e_{\omega_i}\|^2 + k_q\Psi_{q_i} + c_q e_{q_i}\cdot e_{\omega_i},
\end{align*}
where $c_x,c_R,c_q$ are positive constants. 

Let $z_{x_0} = [\|e_{x_0}\|,\|\dot e_{x_0}\|]^T$, $z_{R_0}=[\|e_{R_0}\|,\|e_{\Omega_0}\|]^T$, $z_{q_i}= [\|e_{q_i}\|,\|e_{\omega_i}\|]^T\in\Re^2$. The Lyapunov function satisfies
\begin{align*}
z_{x_0}^T &\underline{P}_{x_0}z_{x_0} + 
z_{R_0}^T \underline{P}_{R_0}z_{R_0} + 
\sum_{i=1}^n z_{q_i}^T \underline{P}_{q_i}z_{q_i}\leq\mathcal{V}\\
&\leq z_{x_0}^T \overline{P}_{x_0}z_{x_0} + 
z_{R_0}^T \overline{P}_{R_0}z_{R_0} + 
\sum_{i=1}^n z_{q_i}^T \overline{P}_{q_i}z_{q_i},
\end{align*}
where the matrices $\underline{P}_{x_0},\underline{P}_{R_0},\underline{P}_{q_i},\underline{P}_{x_0},\underline{P}_{R_0},\underline{P}_{q_i}\in\Re^{2\times 2}$ are given by
\begin{alignat*}{2}
\underline{P}_{x_0}&=\frac{1}{2}\begin{bmatrix} k_{x_0} & -c_x\\-c_x & 1\end{bmatrix},&\;
\overline{P}_{x_0}&=\frac{1}{2}\begin{bmatrix} k_{x_0} & c_x\\c_x & 1\end{bmatrix},\\
\underline{P}_{R_0}&=\frac{1}{2}\begin{bmatrix} 2k_{R_0} & -c_R\overline\lambda\\-c_R\overline\lambda & \underline\lambda\end{bmatrix},&
\overline{P}_{R_0}&=\frac{1}{2}\begin{bmatrix} \frac{2k_{R_0}}{2-\psi_{R_0}} & c_R\overline\lambda\\c_R\overline\lambda & \overline\lambda\end{bmatrix},\\
\underline{P}_{q_i}&=\frac{1}{2}\begin{bmatrix} 2k_{q} & -c_q\\-c_q & 1\end{bmatrix},&
\overline{P}_{q_i}&=\frac{1}{2}\begin{bmatrix} \frac{2k_{q}}{2-\psi_{q_i}} & c_q\\c_q & 1\end{bmatrix},
\end{alignat*}
where $\underline\lambda=\lambda_{m}[J_0]$ and $\overline\lambda=\lambda_{M}[J_0]$. If the constants $c_x,c_{R_0},c_q$ are sufficiently small, all of the above matrices are positive-definite. It follows that the Lyapunov function is positive-definite and decrescent. 

The time-derivative of the Lyapunov function along \refeqn{ddotex0}, \refeqn{doteW0}, and \refeqn{dotewi} is given by
\begin{align*}
\dot{\mathcal{V}} & = -(k_{\dot x_0}-c_x) \|\dot e_{x_0}\|^2 
-c_xk_{x_0} \|e_{x_0}\|^2 -c_xk_{\dot x_0} e_{x_0}\cdot \dot e_{x_0}\nonumber\\
&\quad+ (c_x e_{x_0}+\dot e_{x_0})\cdot Y_x -k_{\Omega_0}\|e_{\Omega_0}\|^2 + c_R \dot e_R\cdot J_0e_{\Omega_0}\nonumber\\
&\quad -c_R k_{R_0}\|e_{R_0}\|^2 + c_R e_{R_0}\cdot ((J_0e_{\Omega_0}+d)^\wedge e_{\Omega_0}-k_{\Omega_0} e_{\Omega_0})\nonumber\\
&\quad + (e_{\Omega_0}+c_R e_{R_0}) \cdot Y_R\\
&\quad +\sum_{i=1}^n -(k_\omega-c_q) \|e_{\omega_i}\|^2 -c_qk_q \|e_{q_i}\|^2 -c_q k_\omega e_{q_i}\cdot e_{\omega_i}.
\end{align*}
Since $\|e_{R_0}\|\leq 1$, $\| \dot e_{R_0}\|\leq \|e_{\Omega_0}\|$ and $\|d\|\leq B$, 
\begin{align}
\dot{\mathcal{V}} & = -(k_{\dot x_0}-c_x) \|\dot e_{x_0}\|^2 
-c_xk_{x_0} \|e_{x_0}\|^2 -c_xk_{\dot x_0} e_{x_0}\cdot \dot e_{x_0}\nonumber\\
&\quad+ (c_x e_{x_0}+\dot e_{x_0})\cdot Y_x -(k_{\Omega_0}-2c_R\overline\lambda)\|e_{\Omega_0}\|^2 \nonumber\\
&\quad -c_R k_{R_0}\|e_{R_0}\|^2 + c_R(k_{\Omega_0}+B) \|e_{R_0}\|\|e_{\Omega_0}\|\nonumber\\
&\quad + (e_{\Omega_0}+c_R e_{R_0}) \cdot Y_R\nonumber\\
&\quad +\sum_{i=1}^n -(k_\omega-c_q) \|e_{\omega_i}\|^2 -c_qk_q \|e_{q_i}\|^2 -c_q k_\omega e_{q_i}\cdot e_{\omega_i}.
\label{eqn:dotV2}
\end{align}
From \refeqn{YxB}, an upper bound of the fourth term of the right-hand side is given by
\begin{align}
\|&(c_x e_{x_0}+\dot e_{x_0})\cdot Y_x\|  \leq  \nonumber\\
&\sum_{i=1}^n\alpha_i\beta( c_xk_{x_0}\|e_{x_0}\|^2 + c_x k_{\dot x_0}\|e_{x_0}\|\|\dot e_{x_0}\| +k_{\dot x_0} \|\dot e_{x_0}\|^2)
\nonumber\\
&+\{c_xB \|e_x\|+(\beta k_{x_0}e_{x_{\max}}+B) \|\dot e_{x_0}\| \} \|e_{q_i}\|\nonumber\\
&+\alpha_i\gamma(c_x\|e_{x_0}\|+\|\dot e_{x_0}\|)(k_{R_0}\|e_{R_0}\|+k_{\Omega_0}\|e_{\Omega_0}\|).
\label{eqn:YxB2}
\end{align}
Similarly, using \refeqn{YRB},
\begin{align}
\|&(c_R e_{R_0}+e_{\Omega_0})\cdot Y_R\|\leq \nonumber\\
&\sum_{i=1}^n \alpha_i\sigma_i( c_R k_{R_0}\|e_{R_0}\|^2 + c_R k_{\Omega_0}\|e_{R_0}\|\|e_{\Omega_0}\| + k_{\Omega_0}\|e_{\Omega_0}\|^2)\nonumber\\
&+\{c_R B\|e_{R_0}\| + (\alpha_0\sigma_ik_{R_0}+B)\|e_{\Omega_0}\|\}\|e_{q_i}\|\nonumber\\
&+\alpha_i\delta_i(c_R \|e_{R_0}\|+\|e_{\Omega_0}\|)(k_{x_0}\|e_{x_0}\|+k_{\dot x_0}\|\dot e_{x_0}\|).
\end{align}
Substituting these into \refeqn{dotV2} and rearranging, $\dot{\mathcal{V}}$ is bounded by
\begin{align*}
\dot{\mathcal{V}} \leq \sum_{i=1}^n - z_i^T W_i z_i,
\end{align*}
where $z=[\|z_{x_0}\|,\, \|z_{R_0}\|,\,\ \|z_{q_i}\|]^T\in\Re^3$, and the matrix $W_i\in\Re^{3\times 3}$ is defined as
\begin{align}
W_i = 
\begin{bmatrix} \lambda_m[W_{x_i}] & -\frac{1}{2}\|W_{xR_i}\| & -\frac{1}{2}\|W_{xq_i}\|\\
-\frac{1}{2}\|W_{xR_i}\| & \lambda_m[W_{R_i}] & -\frac{1}{2}\|W_{Rq_i}\| \\
-\frac{1}{2}\|W_{xq_i}\| & -\frac{1}{2}\|W_{Rq_i}\| & \lambda_m[W_{q_i}]
\end{bmatrix},\label{eqn:Wi}
\end{align}
where the sub-matrices are given by
\begin{gather*}
W_{x_i} = \frac{1}{n}\begin{bmatrix}
c_x k_{x_0} (1-n\alpha_i\beta) & -\frac{c_x k_{\dot x_0}}{2}(1+n\alpha_i\beta)\\
-\frac{c_x k_{\dot x_0}}{2}(1+n\alpha_i\beta) & k_{\dot x_0}(1-n\alpha_i\beta)-c_x 
\end{bmatrix},\\
W_{R_i} = \frac{1}{n}\begin{bmatrix}
c_R k_{R_0} (1-n\alpha_i\sigma_i) & -\frac{c_R }{2}(k_{\Omega_0}+B+n\alpha_i\sigma_i)\\
-\frac{c_R }{2}(k_{\Omega_0}+B+n\alpha_i\sigma_i) & k_{\Omega_0}(1-n\alpha_i\sigma_i)-2c_R\overline\lambda
\end{bmatrix},\\
W_{q_i}=\begin{bmatrix}
c_q k_q & -\frac{c_qk_\omega}{2}\\
-\frac{c_qk_\omega}{2} & k_\omega-c_q\\
\end{bmatrix},\\
W_{xR_i}= \alpha_i\begin{bmatrix}
\gamma c_x k_{R_0}+\delta_i c_Rk_{x_0} 
& \gamma c_x k_{\Omega_0}+\delta_i k_{x_0}\\
\gamma k_{R_0}+\delta_ic_Rk_{\dot x_0}  &  
\gamma k_{\Omega_0} + \delta_i k_{\dot x_0}
\end{bmatrix},\\
W_{xq_i}= \begin{bmatrix}
c_x B & 0\\
\beta k_{x_0}e_{x_{\max}}+B & 0 
\end{bmatrix},\quad
W_{xR_i}= \begin{bmatrix}
c_R B & 0\\
\alpha_0\sigma_i k_{R_0} +B & 0 
\end{bmatrix}.
\end{gather*}
If the constants $c_x,c_R,c_q$ that are independent of the control input are sufficiently small, the matrices $W_{x_i},W_{R_i},W_{q_i}$ are positive-definite. Also, if the error in the direction of the link is sufficiently small relative to the desired trajectory, we can choose the controller gains such that the matrix $W_i$ is positive-definite, which follows that the zero equilibrium of tracking errors is exponentially stable.

\subsection{Proof of Proposition \ref{prop:FDM}}\label{sec:prfFDM}

This proof is based on singular perturbation~\cite{Kha96} and the attitude tracking control system developed in~\cite{LeeLeoPICDC10}. Let $\bar e_{R_i}=\frac{1}{\epsilon}e_{R_i}$. The error dynamics for $\bar e_{R_i},e_{\Omega_i}$ can be written as
\begin{align*}
\epsilon \dot{\bar e}_{R_i} & = \frac{1}{2}(\tr{R_i^T R_{c_i}}I-R_i^T R_{c_i}) e_{\Omega_i},\\
\epsilon \dot e_{\Omega_i} & = J_i^{-1}(-k_R \bar e_{R_i} -k_\Omega e_{\Omega_i}).
\end{align*}
The right-hand side of the above equations has an isolated root of $(\bar e_{R_i},e_{\Omega_i})=(0,0)$, and they correspond to the \textit{boundary-layer} system. And, the origin of the boundary-layer system is exponentially stable according to~\cite[Proposition 1]{LeeLeoPICDC10}.

More explicitly, define a configuration error function on $\SO$ as follows:
\begin{align*}
\Psi_R = \frac{1}{2}\tr{I- R_c^T R}.
\end{align*}
From now on, we drop the subscript $i$ for simplicity, as the subsequent development is identical for all quadrotors. Consider a domain given by $D_R=\{(R,\Omega)\in\SO\times\Re^3\,|\, \Psi_R < \psi_R < 2\}$. Define a Lyapunov function,
\begin{align*}
\mathcal{W} & = \frac{1}{2} e_\Omega\cdot J e_\Omega + \frac{k_R}{\epsilon^2} \Psi_R + \frac{c_3}{\epsilon} e_R\cdot e_\Omega,
\end{align*}
where $c_3$ is a positive constant satisfying
\begin{align*}
c_3 < \min \braces{ \sqrt{k_R\lambda_m(J)},\,\frac{4k_Rk_\Omega\lambda_{m}^2(J)}{k_\Omega^2\lambda_M(J)+4 k_R\lambda_m^2(J)}}.
\end{align*}
We can show that
\begin{align*}
\zeta^T L_1 \zeta \leq \mathcal{W} \leq \zeta^T L_2 \zeta,
\end{align*}
where $\zeta=[\|\bar e_R\|,\, \|e_\Omega\|]\in\Re^2$ and the matrices $L_1,L_2\in\Re^{2\times 2}$ are given by
\begin{align*}
L_1 = \begin{bmatrix} \frac{k_R}{2} & - \frac{c_3}{2} \\ -\frac{c_3}{2} & \frac{\lambda_m(J)}{2}\end{bmatrix},\quad
L_2 = \begin{bmatrix} \frac{k_R}{2-\psi_R} & \frac{c_3}{2} \\ \frac{c_3}{2} & \frac{\lambda_M(J)}{2}\end{bmatrix}.
\end{align*}
The time-derivative of $\mathcal{W}$ can be written as
\begin{align*}
\epsilon\dot{\mathcal{W}} & = (e_\Omega+ c_3 J^{-1}\bar e_R)\cdot (-k_R \bar e_R -k_\Omega e_\Omega)\\
&\quad + k_R \bar e_R\cdot e_\Omega + c_3 \dot e_R \cdot e_\Omega \leq - \zeta^T U \zeta,
\end{align*}
where the matrix $U\in\Re^{2\times 2}$ is 
\begin{align*}
U = \begin{bmatrix} \frac{c_3 k_R}{\lambda_M(J)} & -\frac{c_3 k_\Omega}{2\lambda_m(J)}\\
-\frac{c_3 k_\Omega}{2\lambda_m(J)} & k_\Omega-c_3
  \end{bmatrix}.
\end{align*}
The condition on $c_3$ guarantees that all of matrices $L_1,L_2,U$ are positive-definite. Therefore, the zero equilibrium of the tracking errors $(\bar e_R,e_\Omega)$ is exponentially stable, and the convergence rate is proportional to $\frac{1}{\epsilon}$. 

Next, we consider the \textit{reduced system}, which corresponds to the translational dynamics of the point mass and the rotational dynamics of the links when $R_i=R_{i_c}$. From \refeqn{fi} and \refeqn{b3i}, the control force of quadrotors when $R_i=R_{i_c}$ is given by
\begin{align*}
-f_i \cdot R_i e_3 = (u_i\cdot R_{c_i} e_3) R_{c_i} e_3 = 
(u_i\cdot -\frac{u_i}{\|u_i\|}) -\frac{u_i}{\|u_i\|} = u_i.
\end{align*}
Therefore, the reduced system is given by the controlled dynamics of the simplified model, and from Proposition \ref{prop:SDM}, its origin is exponentially stable. 

Then, according to Tikhonov's theorem~\cite[Thm 9.3]{Kha96}, there exists $\epsilon^* >0$ such that for all $\epsilon<\epsilon^*$, the origin of the full dynamics model is exponentially stable.

\bibliography{/Users/tylee@seas.gwu.edu/Documents/BibMaster}

\begin{thebibliography}{10}
\providecommand{\url}[1]{#1}
\csname url@rmstyle\endcsname
\providecommand{\newblock}{\relax}
\providecommand{\bibinfo}[2]{#2}
\providecommand\BIBentrySTDinterwordspacing{\spaceskip=0pt\relax}
\providecommand\BIBentryALTinterwordstretchfactor{4}
\providecommand\BIBentryALTinterwordspacing{\spaceskip=\fontdimen2\font plus
\BIBentryALTinterwordstretchfactor\fontdimen3\font minus
  \fontdimen4\font\relax}
\providecommand\BIBforeignlanguage[2]{{%
\expandafter\ifx\csname l@#1\endcsname\relax
\typeout{** WARNING: IEEEtran.bst: No hyphenation pattern has been}%
\typeout{** loaded for the language `#1'. Using the pattern for}%
\typeout{** the default language instead.}%
\else
\language=\csname l@#1\endcsname
\fi
#2}}

\bibitem{CicKanJAHS95}
L.~Cicolani, G.~Kanning, and R.~Synnestvedt, ``Simulation of the dynamics of
  helicopter slung load systems,'' \emph{Journal of the American Helicopter
  Society}, vol.~40, no.~4, pp. 44--61, 1995.

\bibitem{BerPICRA09}
M.~Bernard, ``Generic slung load transportation system using small size
  helicopters,'' in \emph{Proceedings of the International Conference on
  Robotics and Automation}, 2009, pp. 3258--3264.

\bibitem{PalCruIRAM12}
I.~Palunko, P.~Cruz, and R.~Fierro, ``Agile load transportation,'' \emph{IEEE
  Robotics and Automation Magazine}, vol.~19, no.~3, pp. 69--79, 2012.

\bibitem{MicFinAR11}
N.~Michael, J.~Fink, and V.~Kumar, ``Cooperative manipulation and
  transportation with aerial robots,'' \emph{Autonomous Robots}, vol.~30, pp.
  73--86, 2011.

\bibitem{MazKonJIRS10}
I.~Maza, K.~Kondak, M.~Bernard, and A.~Ollero, ``Multi-{UAV} cooperation and
  control for load transportation and deployment,'' \emph{Journal of
  Intelligent and Robotic Systems}, vol.~57, pp. 417--449, 2010.

\bibitem{SreLeePICDC13}
K.~Sreenath, T.~Lee, and V.~Kumar, ``Geometric control and differential
  flatness of a quadrotor {UAV} with a cable-suspended load,'' in
  \emph{Proceedings of the IEEE Conference on Decision and Control}, 2013, pp.
  2269--2274.

\bibitem{LeeSrePICDC13}
T.~Lee, K.~Sreenath, and V.~Kumar, ``Geometric control of cooperating multiple
  quadrotor {UAV}s with a suspended load,'' in \emph{Proceedings of the IEEE
  Conference on Decision and Control}, 2013, pp. 5510--5515.

\bibitem{GooLeePACC14}
F.~Goodarzi, D.~Lee, and T.~Lee, ``Geometric stabilization of a quadrotor {UAV}
  with a payload connected by flexible cable,'' in \emph{Proceeding of the
  American Control Conference}, 2014, accepted.

\bibitem{LeeArXiv14a}
\BIBentryALTinterwordspacing
T.~Lee, ``Geometric control of multiple quadrotor {UAVs} transporting a
  cable-suspended rigid body,'' arXiv:1403.3684. [Online]. Available:
  \url{http://arxiv.org/abs/1403.3684}
\BIBentrySTDinterwordspacing

\bibitem{LeeLeoPICDC10}
T.~Lee, M.~Leok, and N.~McClamroch, ``Geometric tracking control of a quadrotor
  {UAV} on {SE(3)},'' in \emph{Proceedings of the IEEE Conference on Decision
  and Control}, 2010, pp. 5420--5425.

\bibitem{Lee08}
T.~Lee, ``Computational geometric mechanics and control of rigid bodies,''
  Ph.D. dissertation, University of Michigan, 2008.

\bibitem{LeeLeoIJNME08}
T.~Lee, M.~Leok, and N.~H. McClamroch, ``Lagrangian mechanics and variational
  integrators on two-spheres,'' \emph{International Journal for Numerical
  Methods in Engineering}, vol.~79, no.~9, pp. 1147--1174, 2009.

\bibitem{GooLeePECC13}
F.~Goodarzi, D.~Lee, and T.~Lee, ``Geometric nonlinear {PID} control of a
  quadrotor {UAV} on {$\SE$},'' in \emph{Proceedings of the European Control
  Conference}, Zurich, July 2013, pp. 3845--3850.

\bibitem{BulLew05}
F.~Bullo and A.~Lewis, \emph{Geometric control of mechanical systems}, ser.
  Texts in Applied Mathematics.\hskip 1em plus 0.5em minus 0.4em\relax New
  York: Springer-Verlag, 2005, vol.~49, modeling, analysis, and design for
  simple mechanical control systems.

\bibitem{Wu12}
T.~Wu, ``Spacecraft relative attitude formation tracking on {SO(3)} based on
  line-of-sight measurements,'' Master's thesis, The George Washington
  University, 2012.

\bibitem{Kha96}
H.~Khalil, \emph{Nonlinear Systems}, 2nd Edition, Ed.\hskip 1em plus 0.5em
  minus 0.4em\relax Prentice Hall, 1996.

\end{thebibliography}
\bibliographystyle{IEEEtran}

\end{document}